\documentclass[fleqn, pdftex]{article} %arXiv

\usepackage{amsmath,amsfonts,latexsym,amssymb,amsthm}
\usepackage{tikz}

\title{{\Large Unary interpretability logics for sublogics of the interpretability logic $\mathbf{IL}$}}
\author{Yuya Okawa}
\date{}

\theoremstyle{plain}
\newtheorem{thm}{Theorem}[section]
\newtheorem{lem}[thm]{Lemma}
\newtheorem{prop}[thm]{Proposition}
\newtheorem{cor}[thm]{Corollary}
\newtheorem{fact}[thm]{Fact}
\newtheorem{prob}[thm]{Problem}
\newtheorem{cl}{Claim}

\theoremstyle{definition}
\newtheorem{defn}[thm]{Definition}

\newtheorem{rem}[thm]{Remark}

\newcommand{\PA}{\mathbf{PA}}

\newcommand{\CL}{\mathbf{CL}}
\newcommand{\IL}{\mathbf{IL}}
\newcommand{\ILM}{\mathbf{ILM}}
\newcommand{\ILP}{\mathbf{ILP}}
\newcommand{\ILS}{\mathbf{IL}_{\mathrm{set}}^{-}}
\newcommand{\il}{\mathbf{il}}
\newcommand{\ilm}{\mathbf{ilm}}
\newcommand{\ilp}{\mathbf{ilp}}
\newcommand{\uI}{\mathbf{I}}
\newcommand{\uR}{\mathbf{uR}}
\newcommand{\rank}{\mathrm{rank}}

\newcommand{\G}[1]{\mathbf{L#1}}
\newcommand{\I}[1]{\mathbf{I#1}}
\newcommand{\J}[1]{\mathbf{J#1}}
\newcommand{\R}[1]{\mathbf{R#1}}
\newcommand{\uJ}[1]{\mathbf{j#1}}
\newcommand{\seq}[1]{\langle#1\rangle}

\begin{document}

\maketitle

\begin{abstract}
De Rijke introduced a unary interpretability logic $\il$, and proved that $\il$ is the unary counterpart of the binary interpretability logic $\IL$. 
In this paper, we find the unary counterparts of the sublogics of $\IL$. %, which were introduced in~\cite{KO20}.   
\end{abstract}

\section{Introduction}

In this section,  we denote recursively enumerable consistent theory extending Peano Arithmetic $\PA$ by $T$. 

The language $\mathcal{L}(\Box, \rhd)$ of interpretability logics is obtained by adding the binary modal operator $\rhd$ to the language of usual propositional modal logics. 
The formula $A \rhd B$ is intended to mean ``$T +B$ is interpretable in $T+A$''. 
The interpretability logic $\IL$ is the base logic to investigate the notion of relative interpretability, and arithmetical completeness was investigated for extensions of it. 
The logics $\ILM$ and $\ILP$ are obtained by adding Montagna's principle $\mathbf{M}$: $A \rhd B \to (A \land \Box C) \rhd (B \land \Box C)$ and Persistence principle $\mathbf{P}$: $A \rhd B \to \Box(A \rhd B)$ to $\IL$, respectively. 
It is known that $\ILM$ and $\ILP$ are arithmetically complete with respect to arithmetical interpretation for suitable theories. 

Several logical properties were investigated on the interpretability logics. 
In~\cite{DeJVis91}, de Jongh and Visser proved that the fixed point property (FPP) holds for the logic $\IL$.  Also, Areces, Hoogland, and de Jongh~\cite{AHD01} proved that the Craig interpolation property (CIP) holds for $\IL$ and $\ILP$. 
On the other hand, Ignatiev~\cite{IgnUn} proved that $\ILM$ does not have CIP.

Guaspari~\cite{Gua79} and Lindstr\"om~\cite{Lin79} independently proved that interpretability is equivalent to $\Pi_{1}$-conservativity for extensions of $\PA$. 
This fact is provable in $\PA$, and hence $\ILM$ is also the logic of $\Pi_{1}$-conservativity (See H\'ajek and Montagna~\cite{HajMon90}). 
Moreover, Ignatiev \cite{Ign91} investigated the logic of $\Gamma$-conservativity, and introduced the conservativity logic $\CL$. 
The logic $\CL$ is the base logic to investigate the notion of partial conservativity and is one of sublogics of $\IL$. 

In~\cite{KO20}, Kurahashi and Okawa introduced other sublogics of $\IL$. 
The weakest logic among them is called $\IL^-$ and other sublogics of $\IL$ are obtained by adding several axiom schemata to $\IL^-$. 
Then, Kurahashi and Okawa investigated their soundness and completeness with respect to relational semantics.

In interpretability logics, the binary modal operator $\rhd$ expresses the behavior of pairs ($\varphi, \psi$) of sentences of arithmetic such that $T +\varphi$ is interpretable in $T+\psi$. 
On the other hand, in a study of sentences of arithmetic concerning relative interpretability and partial conservativity, the existence and properties of $\varphi$ satisfying $T + \varphi$ is interpretable in $T$ or $T + \varphi$ is $\Gamma$-conservative over $T$ have been investigated. 
For example, Feferman~\cite{Fef60} proved that $T + \lnot \mathrm{Con}_{T}$ is interpretable in $T$ where $\mathrm{Con}_{T}$ is a sentence of arithmetic intended to mean ``$T$ is consistent". 
Also, Smory\'nski~\cite{Smo80} proved that $T + \mathrm{Con}_{T}$ is $\Sigma_{1}$-conservative over $T$ if and only if $T$ is $\Sigma_{1}$-sound (For more investigations, see Lindstr\"om~\cite{Lin} and H\'ajek and Pudl\'ak~\cite{HP}). 
Thus, it is natural to focus on $\mathcal{L}(\Box, \rhd)$-formulas which are of the form $\top \rhd A$. 
The language $\mathcal{L}(\Box, \uI)$ of unary interpretability logics is obtained by adding the unary modal operator $\uI$ to the language of usual propositional modal logics where $\uI A$ is an abbreviation of $\top \rhd A$. 
De Rijke~\cite{deR92} investigated unary interpretability logics and introduced the logics $\il$, $\ilm$, and $\ilp$. 
He proved that for any $\mathcal{L}(\Box, \uI)$-formula $A$, $\il \vdash A$ if and only if $\IL \vdash A$, and $\ilm$ and $\ilp$ similarly correspond to $\ILM$ and $\ILP$, respectively. 

Some situations are simplified and changed on unary interpretability logics compared with usual interpretability logics. 
It is known that $\ILM \not \subseteq \ILP$ (See~\cite{Vis88}) and $\ILM$ does not have CIP. 
On the other hand, de Rijke proved that $\ilm \subseteq \ilp$ and $\ilm$ has CIP concerning $\mathcal{L}(\Box, \uI)$-formulas.

In this context, we investigate how unary interpretability logics for the sublogics of $\IL$ behave. 
Our main result is the following (See Theorem~\ref{MainTheorem} in detail): For each of the sublogics $L$ of $\IL$, we find the unary interpretability logic $\ell$, and prove that for any $\mathcal{L}(\Box, \uI)$-formula $A$, $L \vdash A \Longleftrightarrow \ell \vdash A$.

In Section~\ref{Sec:Pre}, we introduce some basics of sublogics of $\IL$. 
In Section~\ref{Sec:Uils}, we introduce the unary interpretability logic $\il^-$ which is intended to correspond to $\IL^-$, and introduce several new axiom schemata. 
Then, we state our main theorem and prove half of it. 
In Section~\ref{Sec:MThm}, we prove remaining half of the main theorem. 
In Section~\ref{Sec:ConFu}, we discuss conclusions and future work. 
We also concern FPP and CIP for unary interpretability logics.

\section{Preliminaries}\label{Sec:Pre}

In this section, we introduce the interpretability logic $\IL$ and several sublogics of it. 

The language $\mathcal{L}(\Box)$ of propositional modal logics consists of countably many propositional variables $p$, $q$, $r$,$\ldots$, the logical constant $\bot$, the connective $\to$, and the modal operator $\Box$.
The language $\mathcal{L}(\Box, \rhd)$ of interpretability logics is obtained by adding the binary modal operator $\rhd$ to $\mathcal{L}(\Box)$. 
Then, $\mathcal{L}(\Box, \rhd)$-formulas are given by 
\[
A::= \bot \mid p \mid A \to A \mid \Box A \mid A \rhd A.
\]
Other symbols $\top$, $\lnot$, $\land$, and $\lor$ are defined in the usual way. 
The formula $\Diamond A$ is an abbreviation of $\lnot \Box \lnot A$.  
The binary modal operator $\rhd$ binds stronger than $\to$, but weaker than $\lnot$, $\land$, $\lor$, $\Box$, and $\Diamond$. 

\begin{defn}
The interpretability logic $\IL$ has the following axiom schemata: 
\begin{itemize}
	\item[$\G{1}$:] All tautologies in the language $\mathcal{L}(\Box, \rhd)$;
	\item[$\G{2}$:] $\Box(A \to B) \to (\Box A \to \Box B)$;
	\item[$\G{3}$:] $\Box(\Box A \to A) \to \Box A$;
	\item[$\J{1}$:] $\Box(A \to B) \to A \rhd B$;
	\item[$\J{2}$:] $(A \rhd B) \land (B \rhd C) \to A \rhd C$;
	\item[$\J{3}$:] $(A \rhd C) \land (B \rhd C) \to (A \lor B) \rhd C$;
	\item[$\J{4}$:] $A \rhd B \to (\Diamond A \to \Diamond B)$;
	\item[$\J{5}$:] $\Diamond A \rhd A$.
\end{itemize}
The inference rules of $\IL$ are Modus Ponens $\dfrac{A \quad A \to B}{B}$ and Necessitation $\dfrac{A}{\Box A}$. 
\end{defn}

In~\cite{Ign91}, Ignatiev introduced the conservativity logic $\CL$ which is a sublogic of $\IL$, and is a base logic to investigate the notion of partial conservativity. 

\begin{defn}
The logic $\CL$ is obtained by removing the axiom schema $\J{5}$ from $\IL$. 
\end{defn}

The logics $\IL$ and $\CL$ are complete with respect to Veltman semantics. 
In this paper, we use  the concept of  $\IL^-\!$-frames (Definition~\ref{Def:IL-frame}) which were originally introduced by Visser \cite{Vis88} as Veltman prestructures. De Jongh and Veltman \cite{deJVel90} proved that $\IL$ is sound and complete with respect to a corresponding class of finite $\IL^- \!$-frames. 
Furthermore, Ignatiev \cite{Ign91} proved that $\CL$ is sound and complete with respect to a corresponding class of them. 

\begin{defn}\label{Def:IL-frame}
A triple $(W, R, \{S_{w}\}_{w \in W})$ is called an \textit{$\IL^-\!$-frame} if it satisfies the following conditions: 
\begin{enumerate}
	\item $W \neq \emptyset$; 
	\item $R$ is a transitive and conversely well-founded binary relation on $W$; 
	\item For each $w \in W$, $S_{w} \subseteq R[w] \times W$ where $R[w] = \{x \in W : w {R} x\}$. 
\end{enumerate}
\end{defn}

\begin{defn}
A quadruple $(W, R, \{S_{w}\}_{w \in W}, \Vdash)$ is called an \textit{$\IL^-\!$-model} if $(W, R, \{S_{w}\}_{w \in W})$ is an $\IL^-\!$-frame and $\Vdash$ is a usual satisfaction relation between $W$ and set of all $\mathcal{L}(\Box, \rhd)$-formulas with the following clauses: 
\begin{enumerate}
	\item $w \Vdash \Box A :\Longleftrightarrow (\forall x \in W)(w {R} x \Rightarrow x \Vdash A)$;
	\item $w \Vdash A \rhd B :\Longleftrightarrow (\forall x \in W)\bigl(w {R} x \, \, \& \, \, x \Vdash A \Rightarrow (\exists y \in W)(x {S_{w}} y \, \, \& \, \, y \Vdash B)\bigr)$.
\end{enumerate}
Let $A$ be any $\mathcal{L}(\Box, \rhd)$-formula and let ${\mathcal F}  = (W, R, \{S_{w}\}_{w \in W})$ be any $\IL^-\!$-frame. 
We say that \textit{$A$ is valid in ${\mathcal F}$} if $w \Vdash A$ for any $\IL^-\!$-model $(W, R, \{S_{w}\}_{w \in W}, \Vdash)$ and any $w \in W$. 
\end{defn}

In \cite{KO20}, other sublogics of $\IL$ were introduced, and their completeness with respect to $\IL^-\!$-frames were investigated. 
The logic $\IL^-$ (see Definition~\ref{Def:IL-}) is the weakest logic among these sublogics. $\IL^-$ is sound and complete with respect to the class of all finite $\IL^-$-frames. 

\begin{defn}\label{Def:IL-}
The logic $\mathbf{IL^{-}}$ has the axiom schemata $\G{1}$, $\G{2}$, $\G{3}$, $\J{3}$, and $\J{6}$: $\Box A \leftrightarrow (\lnot A \rhd \bot)$. 
The inference rules of $\IL^-$ are Modus Ponens, Necessitation, $\R{1}$ $\dfrac{A \to B}{(C \rhd A) \to (C \rhd B)}$, and $\R{2}$ $\dfrac{A \to B}{(B \rhd C) \to (A \rhd C)}$. 
\end{defn}
Other logics introduced in \cite{KO20} are obtained by adding several axiom schemata to $\IL^-$.  
Throughout in this paper, we write $L(\Sigma_{1},\ldots, \Sigma_{k})$ as the logic obtained by adding several axiom schemata $\Sigma_{1}$, $\ldots$, $\Sigma_{k}$ to the logic $L$. 
The following axiom schemata $\J{2}_{+}$ and $\J{4}_{+}$ were introduced in \cite{KO20} and \cite{Vis88}, respectively: 

\begin{itemize}
	\item[$\J{2}_{+}$:] $A \rhd (B \lor C) \land B \rhd C \to (A \rhd C)$;
	\item[$\J{4}_{+}$:] $\Box (A \to B) \to (C \rhd A \to C \rhd B)$.
\end{itemize}

Then, the following facts hold (for proofs, see~\cite{KO20}). 

\begin{fact}\label{p1}
\begin{enumerate}
	\item $\mathbf{IL^{-}} \vdash \Box (A \to B) \to (B \rhd C \to A \rhd C)$;
	\item $\mathbf{IL^{-}}(\J{2}_{+}) \vdash \J{2}$;
	\item $\IL^-(\J{2}_{+}) \vdash \J{4}_{+}$;
	\item $\mathbf{IL^{-}}(\J{1}, \J{2}_{+})$ and $\mathbf{IL^{-}}(\J{1}, \J{2})$ are deductively equivalent to $\CL$;
	\item $\mathbf{IL^{-}}(\J{1}, \J{2}_{+}, \J{5})$ and $\mathbf{IL^{-}}(\J{1}, \J{2}, \J{5})$ are deductively equivalent to $\IL$.
\end{enumerate}
\end{fact}

\begin{rem}
From Facts~\ref{p1}.4 and 5, we always identify $\CL$ and $\IL$ with $\mathbf{IL^{-}}(\J{1}, \J{2}_{+})$ and $\mathbf{IL^{-}}(\J{1}, \J{2}_{+}, \J{5})$, respectively. 
\end{rem}

\begin{fact}\label{FC} Let $\mathcal{F} = (W, R, \{S_{w}\}_{w \in W})$ be any $\mathbf{IL}^{-}\!$-frame. 
\begin{enumerate}
	\item $\J{1}$ is valid in $\mathcal{F}$ $\Longleftrightarrow$ $(\forall x, y \in W) (x {R} y \Rightarrow y {S_{x}} y)$;
	\item $\J{2}_{+}$ is valid in $\mathcal{F}$ $\Longleftrightarrow$ $\J{4}_{+}$ is valid in $\mathcal{F}$ and $(\forall w, x, y, z \in W)(x {S_{w}} y \, \, \& \, \, y {S_{w}} z \Rightarrow x {S_{w}} z)$;
	\item $\J{4}_{+}$ is valid in $\mathcal{F}$ $\Longleftrightarrow$ $(\forall x, y, z \in W)(y {S_{x}} z \Rightarrow x {R} z)$;
	\item $\J{5}$ is valid in $\mathcal{F}$ $\Longleftrightarrow$ $(\forall x, y, z\in W) (x {R} y \, \, \& \, \, y {R} z \Rightarrow y {S_{x}} z)$. 
\end{enumerate}
\end{fact}

\begin{fact}\label{comj15} Let $L$ be any logic appearing in Figure \ref{Fig1}. Then, for any $\mathcal{L}(\Box, \rhd)$-formula $A$, the following are equivalent: 
\begin{enumerate}
	\item $L \vdash A$. 
	\item $A$ is valid in all (finite) $\IL^-\!$-frames in which all axiom schemata of $L$ are valid. 
\end{enumerate}

\begin{figure}[th]
\centering
\begin{tikzpicture}
\node (IL-) at (0,1.5) {$\IL^-$};
\node (IL5) at (3,0) {$\IL^-(\J{5})$};
\node (IL1) at (3,1.5) {$\IL^-(\J{1})$};
\node (IL4) at (3,3){$\IL^-(\J{4}_+)$};
\node (IL15) at (6,0){$\IL^-(\J{1}, \J{5})$};
\node (IL45) at (6,1.5){$\IL^-(\J{4}_+, \J{5})$};
\node (IL14) at (6,3){$\IL^-(\J{1}, \J{4}_+)$};
\node (IL2) at (6,4.5){$\IL^-(\J{2}_+)$};
\node (IL145) at (9,1.5){$\IL^-(\J{1}, \J{4}_+, \J{5})$};
\node (IL25) at (9,3){$\IL^-(\J{2}_+, \J{5})$};
\node (CL) at (9,4.5){$\CL$};
\node (IL) at (12,3){$\IL$};
\draw [-] (IL5)--(IL-);
\draw [-] (IL1)--(IL-);
\draw [-] (IL4)--(IL-);
\draw [-] (IL15)--(IL5);
\draw [-] (IL45)--(IL5);
\draw [-] (IL15)--(IL1);
\draw [-] (IL14)--(IL1);
\draw [-] (IL45)--(IL4);
\draw [-] (IL14)--(IL4);
\draw [-] (IL2)--(IL4);
\draw [-] (IL145)--(IL15);
\draw [-] (IL145)--(IL45);
\draw [-] (IL25)--(IL45);
\draw [-] (IL145)--(IL14);
\draw [-] (CL)--(IL14);
\draw [-] (IL25)--(IL2);
\draw [-] (CL)--(IL2);
\draw [-] (IL)--(IL145);
\draw [-] (IL)--(IL25);
\draw [-] (IL)--(CL);
\end{tikzpicture}
\caption{Sublogics of $\IL$ complete with respect to $\IL^-\!$-frames}
\label{Fig1}
\end{figure}
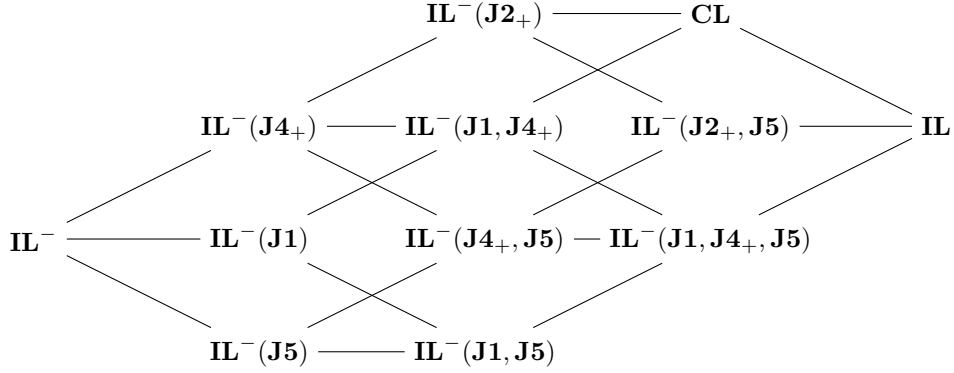
For each of line segments in Figure~\ref{Fig1}, the logic at the left side of the line is a proper sublogic of the logic at the right side of the line. 
\end{fact}

However, there are several sublogics of $\IL$ which are incomplete with respect to Veltman semantics. 
On the other hand, these logics are complete with respect to Verbrugge semantics. 
We use the following concept of $\ILS$-frames which is a general concept of \textit{generalized Veltman frame} introduced by Verbrugge~\cite{Ver92} (See also \cite{Vuk96} or \cite{Joo98}). 

\begin{defn}
A triple $(W, R, \{S_{w}\}_{w \in W})$ is called an \textit{$\ILS$-frame} if it satisfies the following conditions: 
\begin{enumerate}
	\item $W \neq \emptyset$; 
	\item $R$ is a transitive and conversely well-founded binary relation on $W$; 
	\item For each $w \in W$, $S_{w} \subseteq R[w] \times ({\mathcal P}(W) \setminus \{\emptyset\})$ where ${\mathcal P}(W)$ is the power set of $W$; 
	\item (Monotonicity) For any $w, x \in W$ and any $V, U \subseteq W$, if $x {S_{w}} V$ and $V \subseteq U$, then $x {S_{w}} U$. 
\end{enumerate}
\end{defn}

As in the definition of $\IL^-\!$-models, we can define $\ILS$-models $(W, R, \{S_{w}\}_{w \in W}, \mathord \Vdash)$ with the following clause:
\begin{itemize}
	\item[$2.$] $w \Vdash A \rhd B :\Longleftrightarrow (\forall x \in W) \Bigl(w {R} x \, \, \& \, \, x \Vdash A \Rightarrow \bigl(\exists V \in ({\mathcal P}(W) \setminus \emptyset)\bigr)\bigl(x {S_{w}} V \, \, \& \, \, (\forall y \in V) (y \Vdash B)\bigr)\Bigr)$
\end{itemize}

Then, the following facts hold (for proofs, see~\cite{KO20}). 

\begin{fact}\label{gFC} Let $\mathcal{F} = (W, R, \{S_{w}\}_{w \in W})$ be any $\ILS$-frame. 
\begin{enumerate}
	\item $\J{1}$ is valid in $\mathcal{F}$ $\Longleftrightarrow$ $(\forall x, y \in W)(x {R} y \Rightarrow y {S_{x}} \{y\})$;
	\item $\J{2}$ is valid in $\mathcal{F}$ $\Longleftrightarrow$ $\J{4}$ is valid in $\mathcal{F}$ and
\[
(\forall w, x \in W) (\forall V \subseteq W)\left(x {S_{w}}V \, \, \& \, \ (\forall y \in V) \left(y {S_{w}} U_{y}\right) \Rightarrow x {S_{w}} \bigcup_{y \in V} U_{y}\right); 
\]
	\item $\J{4}$ is valid in $\mathcal{F}$ $\Longleftrightarrow$ $(\forall x, y \in W) (\forall V \subseteq W)\bigl(y {S_{x}} V \Rightarrow (\exists z \in V)(x {R} z)\bigr)$;
	\item $\J{4}_{+}$ is valid in $\mathcal{F}$ $\Longleftrightarrow$ $(\forall x, y \in W) (\forall V \subseteq W)\bigl(y {S_{x}} V \Rightarrow y {S_{x}}(V \cap R[x])\bigr)$;
	\item $\J{5}$ is valid in $\mathcal{F}$ $\Longleftrightarrow$ $(\forall x, y, z\in W)(x {R} y \, \, \& \, \, y{R}z \Rightarrow y{S_{x}} \{z\})$.
\end{enumerate}
\end{fact}

\begin{fact}\label{gcomp} 
Let $L$ be any logic appearing in Figure \ref{Fig2}. 
Then, for any $\mathcal{L}(\Box, \rhd)$-formula $A$, the following are equivalent: 
\begin{enumerate}
	\item $L \vdash A$. 
	\item $A$ is valid in all (finite) $\ILS$-frames where all axiom schemata of $L$ are valid. 
\end{enumerate}

\begin{figure}[h]
\centering
\begin{tikzpicture}
\node (IL4) at (0,1.5) {$\IL^-(\J{4})$};
\node (IL14) at (3,0) {$\IL^-(\J{1}, \J{4})$};
\node (IL45) at (3,1.5) {$\IL^-(\J{4}, \J{5})$};
\node (IL2) at (3,3){$\IL^-(\J{2})$};
\node (IL145) at (6,0){$\IL^-(\J{1}, \J{4}, \J{5})$};
\node (IL25) at (6,1.5){$\IL^-(\J{2}, \J{5})$};
\node (IL24) at (6,3){$\IL^-(\J{2}, \J{4}_+)$};
\node (IL245) at (9,1.5){$\IL^-(\J{2}, \J{4}_+, \J{5})$};
\draw [-] (IL14)--(IL4);
\draw [-] (IL45)--(IL4);
\draw [-] (IL2)--(IL4);
\draw [-] (IL145)--(IL14);
\draw [-] (IL145)--(IL45);
\draw [-] (IL25)--(IL45);
\draw [-] (IL25)--(IL2);
\draw [-] (IL24)--(IL2);
\draw [-] (IL245)--(IL25);
\draw [-] (IL245)--(IL24);
\end{tikzpicture}
\caption{Sublogics of $\IL$ complete with respect to $\ILS$-frames}
\label{Fig2}
\end{figure}
The meaning of line segments in Figure~\ref{Fig2} is the same as in Figure~\ref{Fig1}. 
\end{fact}

\section{Sublogics of the unary interpretability logic $\il$}\label{Sec:Uils}

In this section, we investigate sublogics of the unary interpretability logic $\il$ which was introduced by de Rijke \cite{deR92}. 

For any $\mathcal{L}(\Box, \rhd)$-formula $A$, we define the unary modal operator $\uI$ by 
\[
\uI A :\equiv \top \rhd A. 
\] 
The language $\mathcal{L}(\Box, \uI)$ of unary interpretability logics is obtained by adding $\uI$ to $\mathcal{L}(\Box)$. Then, $\mathcal{L}(\Box, \uI)$-formulas are given by 
\[
A::= \bot \mid p \mid A \to A \mid \Box A \mid \uI A. 
\]
De Rijke \cite{deR92} introduced the following unary interpretability logic $\il$, and showed that $\IL$ is a conservative extension of $\il$ concerning $\mathcal{L}(\Box, \uI)$-formulas. 

\begin{defn}
The unary interpretability logic $\mathbf{il}$ has axiom schemata $\G{2}$, $\G{3}$, and the following: 
\begin{itemize}
	\item[$\G{1}'$:] All tautologies in the language $\mathcal{L}(\Box, \uI)$;
	\item[$\I{1}$:] $\uI \Box \bot$;
	\item[$\I{2}$:] $\Box(A \to B) \to (\uI A  \to \uI B)$;
	\item[$\I{3}$:] $\uI(A \lor \Diamond A) \to \uI A$;
	\item[$\I{4}$:] $\uI A \land \Diamond \top \to \Diamond A $.
\end{itemize}
The inference rules of $\il$ are Modus Ponens and Necessitation. 
\end{defn}

\begin{fact}[de Rijke \cite{deR92}]\label{Fact:ilIL}
For any $\mathcal{L}(\Box, \uI)$-formula $A$, 
\[
\IL \vdash A \Longleftrightarrow \il \vdash A.
\] 
\end{fact}

In this paper, we find out the sublogic of $\il$ corresponding to each of the sublogics of $\IL$ satisfying the equivalence relation as in Fact~\ref{Fact:ilIL}.
To find such logics, we newly introduce the following logic $\il^-$ which is intended to correspond to $\IL^-$.

\begin{defn}\label{Def:il-}
The logic $\il^-$ has axiom schemata $\G{1}'$, $\G{2}$, $\G{3}$, and the following: 
\begin{itemize}
	\item[$ \uJ{6}$:] $\Box \bot \leftrightarrow \uI \bot$. 
\end{itemize}
The inference rules of $\il^-$ are Modus Ponens, Necessitation, and the following: 
\begin{itemize}
	\item[$\uR$] $\dfrac{A \to B}{\uI A \to \uI B}$. 
\end{itemize}
\end{defn}

We newly consider the following axiom schemata: 

\begin{itemize}
	\item[$\uJ{1}$:] $\Box A \to \uI A$;
	\item[$\uJ{15}$:] $\Box (A \lor \Diamond A) \to \uI A$;
	\item[$ \uJ{25}$:]  $\Box(A \to \Diamond B) \to (\uI A \to \uI B)$.
\end{itemize}

The rest of this paper is devoted proving the following main theorem. 

\begin{thm}\label{MainTheorem}

For each row on the two tables in Table 1, let $L$ be the logic at the left side of the row and let $\ell$ be the logic at the right side of that. Then, for any $\mathcal{L}(\Box, \uI)$-formula $A$, 
\[
L \vdash A \Longleftrightarrow \ell \vdash A.
\]
\begin{table}[h]\label{Tab1}
\centering
\begin{tabular}[t]{|l|l|}
\hline
$L$ & $\ell$ \\
\hline
\hline
$\IL^-$ & $\il^-$ \\
\cline{1-1}
$\IL^-(\J{5})$ & \\
\hline
$\IL^-(\J{1})$ & $\il^{-}(\uJ{1})$ \\
\hline
$\IL^-(\J{4}_+)$ & $\il^{-}(\I{2})$ \\
\cline{1-1}
$\IL^-(\J{4}_+, \J{5})$ &  \\
\cline{1-1}
$\IL^-(\J{2}_+)$ &  \\
\hline
$\IL^-(\J{1}, \J{5})$ & $\il^{-}(\uJ{15})$ \\
\hline
$\IL^-(\J{1}, \J{4}_+)$ & $\il^{-}(\uJ{1}, \I{2})$ \\
\cline{1-1}
$\CL$ &  \\
\hline
$\IL^-(\J{1}, \J{4}_+, \J{5})$ & $\il^{-}(\uJ{15}, \I{2})$ \\
\hline
$\IL^-(\J{2}_+, \J{5})$ & $\il^{-}(\I{2}, \I{3})$ \\
\hline
$\IL$ & $\il^-(\uJ{1}, \I{2}, \I{3})$ \\ 
\hline
\end{tabular}
\begin{tabular}[t]{|l|l|}
\hline
$L$ & $\ell$ \\
\hline
\hline
$\IL^-(\J{4})$ & $\il^-(\I{4})$ \\
\cline{1-1}
$\IL^-(\J{4}, \J{5})$ & \\
\cline{1-1}
$\IL^-(\J{2})$ & \\
\hline
$\IL^-(\J{1}, \J{4})$ & $\il^{-}(\uJ{1}, \I{4})$ \\
\hline
$\IL^-(\J{1}, \J{4}, \J{5})$ & $\il^{-}(\uJ{15}, \I{4})$ \\
\hline
$\IL^-(\J{2}, \J{5})$ & $\il^{-}(\uJ{25}, \I{4})$ \\
\hline
$\IL^-(\J{2}, \J{4}_{+})$  & $\il^-(\I{2})$ \\
\hline
$\IL^-(\J{2}, \J{4}_{+}, \J{5})$  & $\il^{-}(\I{2}, \uJ{25})$ \\
\hline
\end{tabular}
\caption{The unary interpretability logics correspond to the twenty sublogics of $\IL$}
\end{table}
\end{thm}

\begin{rem}\label{Rem:Collpse}
Theorem~\ref{MainTheorem} tells us that the axiom schemata $\J{5}$, $\J{2}_{+}$, and $\J{2}$ may collapse concerning the unary interpretability logics. 
For example, for any $\mathcal{L}(\Box, \uI)$-formula $A$, we obtain 
\[
\IL^- \vdash A \Longleftrightarrow \IL^-(\J{5}) \vdash A \Longleftrightarrow \il^- \vdash A.
\] 
On the other hand, there are cases where the axiom schemata $\J{5}$, $\J{2}_{+}$, and $\J{2}$ do not collapse. 
For example, for any $\mathcal{L}(\Box, \uI)$-formula $A$, we obtain 
\[
\IL^-(\J{1}) \vdash A \Longleftrightarrow \il^-(\uJ{1}) \vdash A, \, \text{and} \, \, \IL^-(\J{1}, \J{5}) \vdash A \Longleftrightarrow \il^-(\uJ{15}) \vdash A.
\]
However, $\il^-(\uJ{1}) \nvdash \uJ{15}$ holds. 
Let $(W, R, \{S_{w}\}_{w \in W}, \Vdash)$ be the $\IL^-$-model defined as follows:
\begin{itemize}
	\item $W := \{w, x, y\}$; 
	\item $R := \{(w, x), (w, y), (x, y)\}$; 
	\item $S_{w} := \{(x,x), (y, y)\}$, $S_{x}:= \{(y,y)\}$, and $S_{y} :=\emptyset$;
	\item $a \Vdash p :\Longleftrightarrow a = y$. 
\end{itemize}
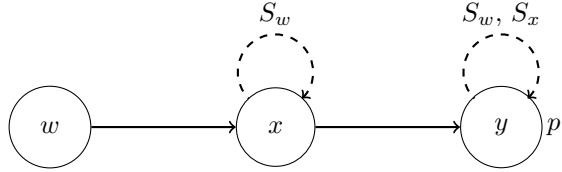
\begin{figure}[th]
\centering
\begin{tikzpicture}

\draw (3, 1.5) node {$S_{w}$};
\draw (6, 1.5) node {$S_{w}$, $S_{x}$};
\draw (6.7, 0) node {$p$};

\node [draw, circle, inner sep=8] (0) at (0,0) {$w$};
\node [draw, circle, inner sep=8] (1) at (3, 0) {$x$};
\node [draw, circle, inner sep=8] (2) at (6, 0) {$y$};

\draw [thick, ->] (0)--(1);
\draw [thick, ->] (1)--(2);

\draw [thick, <-, dashed] (3.35,0.38) arc (-45:225:0.5);
\draw [thick, <-, dashed] (6.35,0.38) arc (-45:225:0.5);

\end{tikzpicture}
\end{figure}
It is easily seen that $(W, R, \{S_{w}\}_{w \in W})$ is an $\IL^-\!$-frame in which $\J{1}$ is valid by Fact~\ref{FC}.1. 
Moreover, $w \nVdash \Box (p \lor \Diamond p) \to \uI p$. 
Thus, we obtain $\IL^-(\J{1}) \nvdash \uJ{15}$ by Fact~\ref{comj15}, and hence $\il^-(\uJ{1}) \nvdash \uJ{15}$. 
\end{rem}

We can easily show the implication $\Longleftarrow$ in Theorem~\ref{MainTheorem} by the following proposition.  

\begin{prop}\label{sound}
\noindent
\begin{enumerate}
	\item Let $L$ be any extension of $\IL^-$ closed under the rule $\R{1}$. Then, $L$ is closed under the rule $\uR$.
	\item $\IL^- \vdash \uJ{6}$.
	\item $\IL^-(\J{1}) \vdash \uJ{1}$.
	\item $\IL^-(\J{4}) \vdash \I{4}$.
	\item $\IL^-(\J{4}_{+}) \vdash \I{2}$.
	\item $\IL^-(\J{1}, \J{5}) \vdash \uJ{15}$.
	\item $\IL^-(\J{2}, \J{5}) \vdash \uJ{25}$.
	\item $\IL^-(\J{2}_{+}, \J{5}) \vdash \I{3}$.
\end{enumerate}
\end{prop}

\begin{proof}
1, 2, 3, 4, 5, and 8 are straightforward. 
We only prove $6$ and $7$.

$6$: Since $\IL^- \vdash \Box(A \lor \Diamond A) \to \Box(\Box \lnot A \to A)$, we have 
\[
\IL^-(\J{1}) \vdash \Box(A \lor \Diamond A) \to (\Box \lnot A \rhd A).
\]
Since $\IL^-(\J{5}) \vdash \Diamond A \rhd A$, we obtain 
\[
\IL^-(\J{1}, \J{5}) \vdash \Box(A \lor \Diamond A) \to \bigl((\Diamond A \lor \Box \lnot A) \rhd A\bigr)
\] 
by $\J{3}$. 
Since $\IL^- \vdash \top \leftrightarrow (\Diamond A \lor \Box \lnot A)$, we have $\IL^- \vdash \uI A \leftrightarrow \bigl((\Diamond A \lor \Box \lnot A) \rhd A\bigr)$ by $\R{2}$. Hence, 
\[
\IL^-(\J{1}, \J{5}) \vdash \Box (A \lor \Diamond A) \to \uI A. 
\]

$7$: By Fact \ref{p1}.1, we have $\IL^- \vdash \Box(A \to \Diamond B) \to \bigl((\Diamond B \rhd B) \to (A \rhd B)\bigr)$. 
Therefore, we obtain 
\[
\IL^-(\J{5}) \vdash \Box(A \to \Diamond B) \to (A \rhd B).
\] 
Since $\IL^-(\J{2}) \vdash (A \rhd B) \land \uI A \to \uI B$, we obtain 
\[
\IL^-(\J{2}, \J{5}) \vdash \Box(A \to \Diamond B) \to (\uI A \to \uI B). \qedhere
\]
\end{proof}

Furthermore, we prove that $\il^-(\uJ{1}, \I{2}, \I{3})$ is deductively equivalent to de Rijke's $\il$ by the following proposition. 

\begin{prop}\label{relation}
\noindent
\begin{enumerate}
	\item $\il^-(\uJ{1}, \I{3}) \vdash \uJ{15}$.
	\item $\il^-(\uJ{15}) \vdash \uJ{1}$.
	\item $\il^-(\uJ{15}) \vdash \I{1}$.
	\item $\il^-(\I{1}, \I{2}) \vdash \uJ{15}$.
	\item $\il^-(\I{2}) \vdash \I{4}$.
	\item $\il \vdash \uJ{6}$.
	\item $\il$ is closed under $\uR$.
\end{enumerate}
\end{prop}

\begin{proof}
$1$: Since $\il^-(\uJ{1}) \vdash \Box (A \lor \Diamond A) \to \uI(A \lor \Diamond A)$, we obtain 
\[
\il^-(\uJ{1}, \I{3}) \vdash \Box (A \lor \Diamond A) \to \uI A.
\]

$2$: This is because $\il^- \vdash \Box A \to \Box(A \lor \Diamond A)$. 

$3$: We have 
\[
\il^-(\uJ{15}) \vdash \Box(\Box \bot \lor \Diamond \Box \bot) \to \uI \Box \bot.
\] 
Since $\il^- \vdash \Diamond \Box \bot \leftrightarrow \lnot \Box \bot$ by $\G{3}$, we obtain 
\[
\il^-(\uJ{15}) \vdash \uI \Box \bot. 
\]

$4$: Since $\il^- \vdash \Box(A \lor \Diamond A) \to \Box(\Box\lnot A \to A)$, we have 
\[
\il^-(\I{2}) \vdash \Box(A \lor \Diamond A) \to (\uI \Box \lnot A \to \uI A).
\]
Since $\il^- \vdash \Box \bot \to \Box\lnot A$, we have $\il^- \vdash \uI \Box \bot \to \uI \Box\lnot A$ by $\uR$. 
Therefore, we obtain
\[
\il^-(\I{1}, \I{2}) \vdash \Box(A \lor \Diamond A) \to \uI A.
\]

$5$: Since $\il^-(\I{2}) \vdash \uI A \land \lnot \uI \bot \to \Diamond A$, we obtain 
\[
\il^-(\I{2}) \vdash \uI A \land \Diamond \top \to \Diamond A
\]
by $\uJ{6}$. 

$6$: By $\I{4}$, we have $\il \vdash \uI \bot \land \Box \top \to \Box \bot$. Therefore, 
\[
\il \vdash \uI \bot \to \Box \bot. 
\]
On the other hand, since $\il \vdash \Box \bot \to \Box (\Box \bot \to \bot)$, we have 
\[
\il \vdash \Box \bot \to (\uI \Box \bot \to \uI \bot)
\]
by $\I{2}$. Thus, by $\I{1}$, we obtain
\[
\il \vdash \Box \bot \to \uI \bot.
\]

$7$: Suppose $\il \vdash A \to B$. Then, $\il \vdash \Box(A \to B)$. 
By $\I{2}$, we conclude $\il \vdash \uI A \to \uI B$. 
\end{proof}

\begin{cor} $\il^-(\uJ{1}, \I{2}, \I{3})$ is deductively equivalent to $\il$. 
\end{cor}

\begin{proof}
By Propositions \ref{relation}.1 and 2, $\il^-(\uJ{1}, \I{2}, \I{3})$ is deductively equivalent to $\il^-(\uJ{15}, \I{2}, \I{3})$. 
By Propositions \ref{relation}.3 and 4, $\il^-(\uJ{15}, \I{2}, \I{3})$ is deductively equivalent to $\il^-(\I{1}, \I{2}, \I{3})$. 
By Propositions \ref{relation}.5, 6, and 7, $\il^-(\I{1}, \I{2}, \I{3})$ is deductively equivalent to $\il$. 
\end{proof}

In the next section, we always identify $\il$ with $\il^-(\uJ{1}, \I{2}, \I{3})$.

\section{The proof of Theorem~\ref{MainTheorem}}\label{Sec:MThm}

In this section, we prove the implication $\Longrightarrow$ in Theorem~\ref{MainTheorem}. 
Our strategy is following: 
Let $L$ and $\ell$ be the logics as in Theorem~\ref{MainTheorem}. 
Then, for any $\mathcal{L}(\Box, \uI)$-formula $A$, we construct an $\IL^-$-frame or an $\ILS$-frame such that $A$ is not valid and all axiom schemta in $L$ are valid in that frames under the supposition $\ell \nvdash A$. 
By Facts~\ref{comj15} and~\ref{gcomp}, we obtain $L \nvdash A$. 

In Subsection~\ref{SubSec:Ppt}, we prepare several lemmas to construct suitable frames. 
In Subsection~\ref{SubSec:MCwrtVF}, we prove Theorem~\ref{MainTheorem} for logics which are complete with respect to $\IL^-$-frames. 
In Subsection~\ref{SubSec:MCwrtGVF}, we prove Theorem~\ref{MainTheorem} for logics which are complete with respect to $\ILS$-frames. 

We denote a logic that contains $\il^-$ (resp.~$\IL^-$) and is closed under the rules of it by $\ell$ (resp.~$L$). 

\subsection{Lemmas for the proof of Theorem~\ref{MainTheorem}}\label{SubSec:Ppt}

In this subsection, we prove several lemmas used in Subsections~\ref{SubSec:MCwrtVF} and~\ref{SubSec:MCwrtGVF}.

\begin{defn}
For any $\mathcal{L}(\Box, \uI)$-formula $A$, we define ${\sim} A$ as follows. 
\begin{eqnarray*}
{\sim} A :\equiv \left\{
\begin{array}{ll}
B & \text{if} \, \, A \, \, \text{is of the form} \, \neg B,\\
\neg A & \text{otherwise.}
\end{array}
\right.
\end{eqnarray*}
\end{defn}

\begin{defn}
Let $\Phi$ be any set of $\mathcal{L}(\Box, \uI)$-formulas. 
Then, $\Phi_{\uI} := \{C : \uI C \in \Phi\}$. 
Moreover, we say that $\Phi$ is an \textit{adequate set} if the following conditions hold: 
\begin{enumerate}
	\item $\Phi$ is closed under taking subformulas and $\sim$-operator;  
	\item $\bot \in \Phi_{\uI}$;
	\item If $C, E \in \Phi_{\uI}$, then $\Box C, \Box (C \lor \Diamond C), \Box(E \to \Diamond C), \Box (E \to C), \Box (E \to C \lor \Diamond C) \in \Phi$. 
\end{enumerate}
\end{defn}

Note that for any finite set $X$ of $\mathcal{L}(\Box, \uI)$-formulas, there exists a finite adequate set $\Phi$ including $X$. 

In this subsection, we fix some finite adequate set $\Phi$. 
For any logic $\ell$, we say that a finite set $X$ of modal formulas is \textit{$\ell$-consistent} if $\ell \nvdash \bigwedge X \to \bot$ where $\bigwedge X$ is a conjunction of all elements of $X$. 
Also, we say that \textit{$\Gamma \subseteq \Phi$ is $\Phi$-maximally $\ell$-consistent} if $\Gamma$ is $\ell$-consistent and either $A \in \Gamma$ or ${\sim}A \in \Gamma$ for any $A \in \Phi$. 
Note that for any $\Phi$-maximally $\ell$-consistent set $\Gamma$, $A \in \Gamma$ if $A \in \Phi$ and $\ell \vdash \bigwedge \Gamma \to A$. 
Moreover, if $X$ is $\ell$-consistent, then there exists some $\Phi$-maximally $\ell$-consistent set $\Gamma$ such that $X \subseteq \Gamma$. 

\begin{defn}
Let $\ell$ be any logic. Then, we define
\[
K_{\ell} := \{\Gamma \subseteq \Phi : \Gamma \, \, \text{is} \, \, \Phi\text{-maximally} \, \, \ell\text{-consistent}\}.
\]
\end{defn}

\begin{defn}
Let $\ell$ be any logic and let $\Gamma$, $\Delta \in K_{\ell}$. 
Then, 
\begin{eqnarray*}
\Gamma \prec \Delta :\Longleftrightarrow  
\begin{array}{ll}
1. \, \{B, \Box B : \Box B \in \Gamma\} \subseteq \Delta,\\
2. \, \text{There exists} \, \, \Box C \in \Delta \, \, \text{such that} \, \, \Box C \notin \Gamma. 
\end{array}
\end{eqnarray*}
\end{defn}
Note that $\prec$ is transitive and irreflexive. 

\begin{lem}\label{lem0} 
Let $\ell$ be any logic and $\Gamma \in K_{\ell}$.  
If $\lnot \Box C \in \Gamma$, then there exists $\Delta \in K_{\ell}$ such that $\Gamma \prec \Delta$ and ${\sim}C, \Box C \in \Delta$.
\end{lem}

\begin{proof}
It suffices to show that the following set $X$ is $\ell$-consistent. 
\[
X := \{B, \Box B : \Box B \in \Gamma\} \cup \{\Box C, {\sim}C\}. 
\]
Towards a contradiction, we suppose $\ell \vdash \lnot \bigwedge X$. 
\begin{eqnarray*}
\ell &\vdash& \bigwedge\{B, \Box B : \Box B \in \Gamma\} \to (\Box C \to C),\\
\ell &\vdash& \bigwedge \Gamma \to \Box(\Box C \to C),\\
\ell &\vdash& \bigwedge \Gamma \to \Box C.
\end{eqnarray*}
This contradicts $\lnot \Box C \in \Gamma$. 
\end{proof}

\begin{lem}\label{lem1} 
Let $\ell$ be any logic, let $\Gamma \in K_{\ell}$ and let $\lnot \uI C \in \Gamma$. 
Then, the following conditions hold:
\begin{enumerate}
	\item There exists $\Delta \in K_{\ell}$ such that $\Gamma \prec \Delta$ and $\Box \bot \in \Delta$; 
	\item If $\ell \vdash \uJ{1}$, then there exists $\Delta \in K_{\ell}$ such that $\Gamma \prec \Delta$ and ${\sim}C \in \Delta$; 
	\item If $\ell \vdash \uJ{15}$, then there exists $\Delta \in K_{\ell}$ such that $\Gamma \prec \Delta$ and ${\sim}C, \Box \lnot C \in \Delta$. 
\end{enumerate}
\end{lem}

\begin{proof}
1: We show $\Box \bot \notin \Gamma$. 
Suppose $\Box \bot \in \Gamma$. Then, $\ell \vdash \bigwedge \Gamma \to \uI \bot$  by $ \uJ{6}$, and hence we obtain $\ell \vdash \bigwedge \Gamma \to \uI C$. 
This contradicts $\lnot \uI C \in \Gamma$. 
Thus, $\Box \bot \notin \Gamma$. 
By Lemma~\ref{lem0}, there exists $\Delta \in K_{\ell}$ such that $\Gamma \prec \Delta$ and $\Box \bot \in \Delta$. 

2: Assume $\ell \vdash \uJ{1}$. 
If $\Box C \in \Gamma$, then $\uI C \in \Gamma$ by $ \uJ{1}$. 
Thus, $\Box C \notin \Gamma$. 
By Lemma~\ref{lem0}, there exists $\Delta \in K_{\ell}$ such that $\Gamma \prec \Delta$ and ${\sim}C \in \Delta$. 

3: Assume $\ell \vdash \uJ{15}$. 
If $\Box (C \lor \Diamond C) \in \Gamma$, then $\uI C \in \Gamma$ by $ \uJ{15}$. 
Thus, $\Box (C \lor \Diamond C) \notin \Gamma$. 
By Lemma~\ref{lem0}, there exists $\Delta \in K_{\ell}$ such that $\Gamma \prec \Delta$ and ${\sim}(C \lor \Diamond C) \in \Delta$. 
Then, $\ell \vdash \bigwedge \Delta \to {\sim}C$ and $\ell \vdash \bigwedge \Delta \to \Box \lnot C$. 
Hence, ${\sim}C, \Box \lnot C \in \Delta$. 
\end{proof}

\begin{lem}\label{lem2} 
Let $\ell$ be any logic, let $\Gamma \in K_{\ell}$, and let $\uI C, \lnot \uI D \in \Gamma$. 
Then, the following conditions hold:
\begin{enumerate}
	\item There exists $\Delta \in K_{\ell}$ such that $C, {\sim}D \in \Delta$. 
	\item If $\ell \vdash \uJ{25}$, then there exists $\Delta \in K_{\ell}$ such that $\Gamma \prec \Delta$ and $C, \Box \lnot D \in \Delta$. 
	\item If $\ell \vdash \I{2}$, then there exists $\Delta \in K_{\ell}$ such that $\Gamma \prec \Delta$ and $C, {\sim}D \in \Delta$. 
	\item If $\ell \vdash \I{2}, \I{3}$, then there exists $\Delta \in K_{\ell}$ such that $\Gamma \prec \Delta$ and $C, {\sim}D, \Box \lnot D \in \Delta$. 
\end{enumerate}
\end{lem}

\begin{proof}
1: It suffices to show that the following set $X$ is consistent with $\ell$. 
\[
X := \{C, {\sim}D\}. 
\]
Towards a contradiction, we suppose $X$ is inconsistent with $\ell$. 
Then, 
\begin{eqnarray*}
\ell &\vdash& C \to D,\\
\ell &\vdash& \uI C \to \uI D,\\
\ell &\vdash& \bigwedge \Gamma \to \uI D. 
\end{eqnarray*}
This contradicts $\lnot \uI D \in \Gamma$. 

2: Assume $\ell \vdash \uJ{25}$. 
If $\Box (C \to \Diamond D) \in \Gamma$, then $\ell \vdash \bigwedge \Gamma \to (\uI C \to \uI D)$ by $ \uJ{25}$, and hence $\uI D \in \Gamma$. 
Thus, $\Box (C \to \Diamond D) \notin \Gamma$. 
By Lemma~\ref{lem0}, there exists $\Delta \in K_{\ell}$ such that $\Gamma \prec \Delta$ and ${\sim}(C \to \Diamond D) \in \Delta$. 
Hence, $C, \Box \lnot D \in \Delta$. 

3: Assume $\ell \vdash \I{2}$. 
If $\Box (C \to D) \in \Gamma$, then $\ell \vdash \bigwedge \Gamma \to (\uI C \to \uI D)$ by $\I{2}$, and hence $\uI D \in \Gamma$. 
Thus, $\Box (C \to D) \notin \Gamma$. 
By Lemma~\ref{lem0}, there exists $\Delta \in K_{\ell}$ such that $\Gamma \prec \Delta$ and ${\sim}(C \to D) \in \Delta$. 
Hence, $C, {\sim}D \in \Delta$. 

4: Assume $\ell \vdash \I{2}, \I{3}$. 
If $\Box (C \to D \lor \Diamond D) \in \Gamma$, then $\ell \vdash \bigwedge \Gamma \to (\uI C \to \uI D)$ by $\I{2}$ and $\I{3}$, and hence $\uI D \in \Gamma$. 
Thus, $\Box (C \to D \lor \Diamond D) \notin \Gamma$. 
By Lemma~\ref{lem0}, there exists $\Delta \in K_{\ell}$ such that $\Gamma \prec \Delta$ and ${\sim}(C \to D \lor \Diamond D) \in \Delta$. 
Hence, $C, {\sim}D, \Box \lnot D \in \Delta$. 
\end{proof}

\begin{lem}\label{lem3}
Let $\ell$ be any logic having $\I{4}$ and let $\Gamma, \Delta \in K_{\ell}$. 
If $\uI C \in \Gamma$ and $\Gamma \prec \Delta$, then there exists $\Theta \in K_{\ell}$ such that $\Gamma \prec \Theta$ and $C \in \Theta$. 
\end{lem}

\begin{proof} 
Since $\ell \vdash \bigwedge \Gamma \to (\Diamond \top \to \Diamond C)$ by $\uI C \in \Gamma$ and $\I{4}$, if $\Box \lnot C \in \Gamma$, then $\ell \vdash \bigwedge \Gamma \to \Box \bot$. 
Since $\Gamma \prec \Delta$, we have $\Box \lnot C \notin \Gamma$. 
By Lemma~\ref{lem0}, there exists $\Theta \in K_{\ell}$ such that $\Gamma \prec \Theta$ and $C \in \Theta$. 
\end{proof}

\subsection{The unary interpretability logics for the sublogics of $\IL$ which are complete with respect to $\IL^-$-frames}\label{SubSec:MCwrtVF}

In this subsection, we consider the twelve sublogics of $\IL$ which are complete with respect to $\IL^- \!$-frames. 
Firstly, we give a proof of the implication $\Longrightarrow$ in Theorem~\ref{MainTheorem} for the ten logics $\IL^-$, $\IL^-(\J{5})$, $\IL^-(\J{1})$, $\IL^{-}(\J{4}_{+})$, $\IL^{-}(\J{4}_{+}, \J{5})$, $\IL^{-}(\J{2}_{+})$, $\IL^-(\J{1}, \J{5})$, $\IL^-(\J{1}, \J{4}_{+})$, $\CL$, and $\IL^-(\J{1}, \J{4}_{+}, \J{5})$.

\begin{thm}\label{ilcomp}
Let $L \in $$\{\IL^-(\J{5})$, $\IL^-(\J{1})$, $\IL^{-}(\J{4}_{+}, \J{5})$, $\IL^{-}(\J{2}_{+})$, $\IL^-(\J{1}, \J{5})$, $\CL$, $\IL^-(\J{1}, \J{4}_{+}, \J{5})\}$ and let $\ell$ be the following logic; 
\begin{eqnarray*}
\ell := \left\{
\begin{array}{ll}
\il^- & \text{if} \, \, L = \IL^-(\J{5}),\\
\il^-(\uJ{1}) & \text{if} \, \, L = \IL^-(\J{1}),\\
\il^-(\I{2}) & \text{if} \, \, L = \IL^{-}(\J{4}_{+}, \J{5}) \, \, \text{or} \, \, L = \IL^{-}(\J{2}_{+}),\\
\il^-(\uJ{15}) & \text{if} \, \, L = \IL^-(\J{1}, \J{5}),\\
\il^-(\uJ{1}, \I{2}) & \text{if} \, \, L = \CL,\\
\il^-(\uJ{15}, \I{2}) & \text{if} \, \, L = \IL^-(\J{1}, \J{4}_{+}, \J{5}).
\end{array}
\right.
\end{eqnarray*}
For any $\mathcal{L}(\Box, \uI)$-formula $A$, if $L \vdash A$, then $\ell \vdash A$. 
\end{thm}

\begin{proof}

Suppose $\ell \nvdash A$. 
By Fact~\ref{comj15}, it suffices to show that there exist an $\IL^- \!$-model $(W, R, \{S_{x}\}_{x \in W}, \Vdash)$ in which all axiom schemata in $L$ are valid and $w \in W$ such that $w \nVdash A$. 
Let $\Phi$ be any finite adequate set including $A$ and let $\Gamma_{0} \in K_{\ell}$ such that ${\sim}A \in \Gamma_{0}$. 
Then, we define $\mathcal{F} := (W, R, \{S_{x}\}_{x \in W})$ as follows:
\begin{itemize}
	\item $W := \{(\Gamma, B): \Gamma \in K_{\ell}, B \in \Phi_{\uI}\}$; 
	\item $(\Gamma, B) {R} (\Delta, C) :\Longleftrightarrow \Gamma \prec \Delta$; 
	\item $(\Delta, C) {S_{(\Gamma, B)}} (\Theta, D) :\Longleftrightarrow \Gamma \prec \Delta$ and the following conditions that depend on $L$ hold. 
\begin{itemize}
	\item $L = \IL^-(\J{5})$: If $\lnot \uI C \in \Gamma$ and $\Box \lnot C \in \Delta$, then ${\sim}C \in \Theta$.
	\item $L = \IL^-(\J{1})$: If $\lnot \uI C \in \Gamma$ and ${\sim}C \in \Delta$, then ${\sim}C \in \Theta$.
	\item $L = \IL^{-}(\J{4}_{+}, \J{5})$: $\Gamma \prec \Theta$ and if $\lnot \uI C \in \Gamma$ and $\Box \lnot C \in \Delta$, then ${\sim}C \in \Theta$. 
	\item $L = \IL^{-}(\J{2}_{+})$: $\Gamma \prec \Theta$ and if $\lnot \uI C \in \Gamma$ then ${\sim}C \in \Theta$ and $C \equiv D$. 
	\item $L = \IL^-(\J{1}, \J{5})$: If $\lnot \uI C \in \Gamma$ and ${\sim}C, \Box \lnot C \in \Delta$, then ${\sim}C \in \Theta$. 
	\item $L = \CL$: $\Gamma \prec \Theta$ and if $\lnot \uI C \in \Gamma$ and ${\sim} C \in \Delta$, then ${\sim}C \in \Theta$ and $C \equiv D$. 
	\item $L = \IL^-(\J{1}, \J{4}_{+}, \J{5})$: $\Gamma \prec \Theta$ and if $\lnot \uI C \in \Gamma$ and ${\sim} C, \Box \lnot C \in \Delta$, then ${\sim}C \in \Theta$.
\end{itemize} 
\end{itemize}

Since $(\Gamma_{0}, \bot) \in W$ and $\Phi$ is finite, $W \neq \emptyset$ and $W$ is finite. 
Also, $R$ is irreflexive and transitive, and $(\Delta, C) {S_{(\Gamma, B)}} (\Theta, D)$ implies $(\Gamma, B) {R} (\Delta, C)$. 
Therefore, $\mathcal{F}$ is a finite $\IL^{-}$-frame. 
We show that all axiom schemata of $L$ are valid in $\mathcal{F}$. 
We distinguish the following four cases. 

\begin{itemize}
	\item $L = \IL^-(\J{5})$:  
We show that $\J{5}$ is valid in $\mathcal{F}$. 
Assume $(\Gamma, B) {R} (\Delta, C)$ and $(\Delta, C) {R} (\Theta, D)$. 
Then, $\Gamma \prec \Delta$. 
Moreover, if $\lnot \uI C \in \Gamma$ and $\Box \lnot C \in \Delta$, then ${\sim}C \in \Theta$ because $\Delta \prec \Theta$. 
Thus, $(\Delta, C) {S_{(\Gamma, B)}} (\Theta, D)$. 
By Fact \ref{FC}.4, $\J{5}$ is valid in $\mathcal{F}$. 

	\item $L = \IL^-(\J{1})$:  
We show that $\J{1}$ is valid in $\mathcal{F}$. 
Assume $(\Gamma, B) {R} (\Delta, C)$. 
Then, we have $(\Delta, C) {S_{(\Gamma, B)}} (\Delta, C)$ by the definition of $S_{(\Gamma, B)}$. 
By Fact \ref{FC}.1,  $\J{1}$ is valid in $\mathcal{F}$. 

	\item $L = \IL^-(\J{2}_{+})$: 
We show that $\J{2}_{+}$ is valid in $\mathcal{F}$. 
If $(\Delta, C) {S_{(\Gamma, B)}} (\Theta, D)$, then $(\Gamma, B) {R} (\Theta, D)$ by the definition of  $S_{(\Gamma, B)}$. Thus, $\J{4}_{+}$ is valid in $\mathcal{F}$ by Fact~\ref{FC}.3.
Suppose $(\Delta_{1}, C_{1}) {S_{(\Gamma, B)}} (\Delta_{2}, C_{2})$ and $(\Delta_{2}, C_{2}) {S_{(\Gamma, B)}} (\Delta_{3}, C_{3})$. 
Then, we have $\Gamma \prec \Delta_{1}$ and $\Gamma \prec \Delta_{3}$. 
Also, if $\lnot \uI C_{1} \in \Gamma$, then $C_{1} \equiv C_{2}$, and hence we obtain ${\sim} C_{1} \in \Delta_{3}$ and $C_{1} \equiv C_{3}$. 
Therefore, we obtain $(\Delta_{1}, C_{1}) {S_{(\Gamma, B)}} (\Delta_{3}, C_{3})$. 
By Fact \ref{FC}.2,  $\J{2}_{+}$ is valid in $\mathcal{F}$. 

	\item Other cases are proved in a similar way as above. 
\end{itemize}

Let $(W, R, \{S_{x}\}_{x \in W}, \Vdash)$ be an $\IL^-$-model  with 
\[
(\Gamma, B) \Vdash p \Longleftrightarrow p \in \Gamma
\] 
for any $(\Gamma, B) \in W$. 
We prove the following claim. 

\begin{cl}[Truth Lemma]
Let $A' \in \Phi$ and let $(\Gamma, B) \in W$. Then, 
\[
A' \in \Gamma \Longleftrightarrow (\Gamma, B) \Vdash A'.
\]
\end{cl}

\begin{proof}
We prove by induction on the construction of $A'$. 
We only prove the cases $A' \equiv \Box C$ and $A' \equiv \uI C$. 

\begin{itemize}
	\item[Case~1]: $A' \equiv \Box C$.

($\Longrightarrow$): Suppose $\Box C \in \Gamma$. 
Let $(\Delta, D) \in W$ such that $(\Gamma, B) {R} (\Delta, D)$. 
Then, $\Gamma \prec \Delta$, and hence $C \in \Delta$. 
By induction hypothesis, $(\Delta, D) \Vdash C$. 
Therefore, $(\Gamma, B) \Vdash \Box C$

($\Longleftarrow$): Suppose $\Box C \notin \Gamma$. 
By Lemma~\ref{lem0}, there exists $\Delta \in K_{\ell}$ such that $\Gamma \prec \Delta$ and ${\sim}C \in \Delta$. 
Then, we obtain $(\Delta, \bot) \in W$, $(\Gamma, B) {R} (\Delta, \bot)$, and $(\Delta, \bot) \nVdash C$ by induction hypothesis. 
Therefore, $(\Gamma, B) \nVdash \Box C$.

	\item[Case~2]:  $A' \equiv \uI C$. 

($\Longrightarrow$): Suppose $\uI C \in \Gamma$. 
We show $(\Gamma, B) \Vdash \uI C$. 
Let $(\Delta, D) \in W$ such that $(\Gamma, B) {R} (\Delta, D)$. 
We distinguish the following two cases. 
\begin{itemize}
	\item Suppose $\lnot \uI D \in \Gamma$. 
Since $\uI C \in \Gamma$,  there exists $\Theta \in K_{\ell}$ such that $C, {\sim} D \in \Theta$ by Lemma \ref{lem2}.1. 
Moreover, if $\ell$ has $\I{2}$, then $\Gamma \prec \Theta$ holds by Lemma \ref{lem2}.3. 
Therefore, we obtain $(\Delta, D) {S_{(\Gamma, B)}} (\Theta, D)$. 

	\item Otherwise. 
Since $\Gamma \prec \Delta$, we have $\Box \bot \notin \Gamma$. 
Hence, $\uI \bot \notin \Gamma$ by $\uJ{6}$. 
Since $\uI C \in \Gamma$, there exists $\Theta \in K_{\ell}$ such that $C \in \Theta$ by Lemma \ref{lem2}.1. 
Moreover, if $\ell$ has $\I{2}$, then $\Gamma \prec \Theta$ holds by Lemma \ref{lem2}.3. 
Therefore, we obtain $(\Delta, D) {S_{(\Gamma, B)}} (\Theta, D)$. 
\end{itemize}
In any case, $(\Theta, D) \Vdash C$ by induction hypothesis, and hence $(\Gamma, B) \Vdash \uI C$. 

($\Longleftarrow$): Suppose $\uI C \notin \Gamma$. 
We show $(\Gamma, B) \nVdash \uI C$. 
We take $\Delta \in K_{\ell}$ for each of the following three cases. 
\begin{itemize}
	\item $\ell$ has $\uJ{15}$: 
There exists $\Delta \in K_{\ell}$ such that $\Gamma \prec \Delta$ and ${\sim} C, \Box \lnot C \in \Delta$ by Lemma \ref{lem1}.3. 

	\item $\ell$ has $\uJ{1}$: 
There exists $\Delta \in K_{\ell}$ such that $\Gamma \prec \Delta$ and $ {\sim} C \in \Delta$ by Lemma \ref{lem1}.2. 

	\item Otherwise: 
Since $\uI C \notin \Gamma$, there exists $\Delta \in K_{\ell}$ such that $\Gamma \prec \Delta$ and $\Box \bot \in \Delta$ by Lemma \ref{lem1}.1, and hence $\Box \lnot C \in \Delta$. 
\end{itemize}

In any cases, we have $(\Gamma, B) {R} (\Delta, C)$. 
Let $(\Theta, E) \in W$ such that $(\Delta, C) {S_{(\Gamma, B)}} (\Theta, E)$. 
Then, ${\sim}C \in \Theta$ by the definition of $S_{(\Gamma, B)}$, and hence $(\Theta, E) \nVdash C$ by induction hypothesis. 
Therefore, we obtain $(\Gamma, B) \nVdash \uI D$. \qedhere
\end{itemize}
\end{proof}

Since $(\Gamma_{0}, \bot) \in W$ and ${\sim}A \in \Gamma_{0}$, we obtain $(\Gamma_{0}, \bot) \nVdash A$ by Claim 1. 
\end{proof}

By Theorem~\ref{ilcomp} and the implication $\Longleftarrow$ in Thenorem~\ref{MainTheorem}, we obtain the following corollary. 

\begin{cor}\label{cor1}
Let $A$ be any $\mathcal{L}(\Box, \uI)$-formula. 
\begin{enumerate}
	\item Let $L$ be any logic such that $\IL^- \subseteq L \subseteq \IL^{-}(\J{5})$. 
Then, $L \vdash A$ iff $\il^- \vdash A$. 
	\item Let $L$ be any logic such that $\IL^-(\J{4}_{+}) \subseteq L \subseteq \IL^-(\J{4}_{+}, \J{5})$ or $\IL^-(\J{4}_{+}) \subseteq L \subseteq \IL^-(\J{2}_{+})$. 
Then, $L \vdash A$ iff $\il^-(\I{2}) \vdash A$. 
	\item Let $L$ be any logic such that $\IL^-(\J{1}, \J{4}_{+}) \subseteq L \subseteq \CL$. 
Then, $L \vdash A$ iff $\il^-(\uJ{1}, \I{2}) \vdash A$. 
\end{enumerate}
\end{cor}

Next, we consider the remaining two logics $\IL^-(\J{2}_{+}, \J{5})$ and $\IL$. 

\begin{rem}
De Rijke already proved that $\il$ is the unary interpretability logic of $\IL$. 
In this paper, we give another proof that is slightly different from de Rijke's one.
\end{rem}

We define several notations. 
Let $\tau$ and $\sigma$ be any sequences of $\mathcal{L}(\Box. \uI)$-formulas. 
$\tau \subseteq \sigma$ denotes that $\tau$ is an initial segment of $\sigma$. 
$\tau \subset \sigma$ denotes that $\tau \subseteq \sigma$ and $\tau \neq \sigma$. 
$|\tau|$ is the length of $\tau$. 
$\tau * \seq{A}$ is the sequence obtained from $\tau$ by concatenating $A$ at the last element. 
For example, $\seq{A, B} \subset \seq{A, B, C}$, $|\seq{A, B, C}| = 3$ and $\seq{A, B, C}*\seq{D} = \seq{A, B, C, D}$. 

\begin{thm}\label{ilcomp2a5}
Let $L \in \{\IL^-(\J{2}_{+}, \J{5}) , \IL\}$ and let $\ell$ be the following logic; 
\begin{eqnarray*}
\ell := \left\{
\begin{array}{ll}
\il^-(\I{2}, \I{3}) & \text{if} \, \, L = \IL^-(\J{2}_{+}, \J{5}),\\
\il^-(\uJ{1}, \I{2}, \I{3}) & \text{if} \, \, L = \IL. 
\end{array}
\right.
\end{eqnarray*}
For any $\mathcal{L}(\Box, \uI)$-formula $A$, if $L \vdash A$, then $\ell \vdash A$. 
\end{thm}

\begin{proof}

Suppose $\ell \nvdash A$. 
We find an $\IL^- \!$-model $(W, R, \{S_{x}\}_{x \in W}, \Vdash)$ in which all axiom schemata in $L$ are valid and $w \in W$ with $w \nVdash A$. 
Let $\Phi$ be any finite adequate set including $A$ and let $\Gamma_{0} \in K_{\ell}$ with ${\sim}A \in \Gamma_{0}$. 
For each $\Gamma \in K_{\ell}$, we define $\rank(\Gamma) := \sup\{\rank(\Delta) + 1 : \Gamma \prec \Delta\}$ where $\sup \emptyset = 0$. 
Then, we define $\mathcal{F} = (W, R, \{S_{x}\}_{x \in W})$ as follows:
\begin{itemize}
	\item $W := \{(\Gamma, \tau): \Gamma \in K_{\ell}$ and $\tau$ is a finite sequence of elements of $\Phi_{\uI}$ with $\rank(\Gamma) + |\tau| \leq \rank(\Gamma_{0})\}$; 
	\item $(\Gamma, \tau) {R} (\Delta, \sigma) :\Longleftrightarrow \Gamma \prec \Delta$ and $\tau \subset \sigma$; 
	\item $(\Delta, \sigma) {S_{(\Gamma, \tau)}} (\Theta, \rho) :\Longleftrightarrow (\Gamma, \tau) {R} (\Delta, \sigma)$, $(\Gamma, \tau) {R} (\Theta, \rho)$, and the following conditions that depend on $L$ hold. 
\begin{itemize}
	\item $L = \IL^-(\J{2}_{+}, \J{5})$: If $\lnot \uI C \in \Gamma$, $\Box \lnot C \in \Delta$, and $\tau*\seq{C} \subseteq \sigma$, then ${\sim}C, \Box \lnot C \in \Theta$ and $\tau*\seq{C} \subseteq \rho$. 
	\item $L = \IL$:  If $\lnot \uI C \in \Gamma$, ${\sim}C, \Box \lnot C \in \Delta$, and $\tau*\seq{C} \subseteq \sigma$, then ${\sim}C, \Box \lnot C \in \Theta$ and $\tau*\seq{C} \subseteq \rho$. 
\end{itemize} 
\end{itemize}

Since $(\Gamma_{0}, \epsilon) \in W$ where $\epsilon$ is the empty sequence, we have $W \neq \emptyset$. 
Thus, $\mathcal{F}$ is obviously a finite $\IL^{-}$-frame. 
We show that all axiom schemata of $L$ are valid in $\mathcal{F}$. 
In any case, $\J{4}_{+}$ is obviously valid in $\mathcal{F}$ by Fact~\ref{FC}.3. 
We distinguish the following two cases. 

\begin{itemize}
	\item $L = \IL^-(\J{2}_{+}, \J{5})$:  
Firstly, we show that $\J{2}_{+}$ is valid in $\mathcal{F}$.  
Assume $(\Delta_{1}, \sigma_{1}) {S_{(\Gamma, \tau)}} (\Delta_{2}, \sigma_{2})$ and $(\Delta_{2}, \sigma_{2}) {S_{(\Gamma, \tau)}} (\Delta_{3}, \sigma_{3})$. 
Then, $(\Gamma, \tau) {R} (\Delta_{1}, \sigma_{1})$ and $(\Gamma, \tau) {R} (\Delta_{3}, \sigma_{3})$ hold. 
Also, if $\lnot \uI C \in \Gamma$, $\Box \lnot C \in \Delta_{1}$, and $\tau*\seq{C} \subseteq \sigma_{1}$, then we obtain ${\sim}C, \Box \lnot C \in \Delta_{2}$ and $\tau*\seq{C} \subseteq \sigma_{2}$, and hence ${\sim}C, \Box{\sim}C \in \Delta_{3}$ and $\tau*\seq{C} \subseteq \sigma_{3}$. 
Thus, we obtain $(\Delta_{1}, \sigma_{1}) {S_{(\Gamma, \tau)}} (\Delta_{3}, \sigma_{3})$. 
By Fact \ref{FC}.2, $\J{2}_{+}$ is valid in $\mathcal{F}$.  

Secondly, we show that $\J{5}$ is valid in $\mathcal{F}$. 
Assume $(\Gamma, \tau) {R} (\Delta, \sigma)$ and $(\Delta, \sigma) {R} (\Theta, \rho)$. 
Then, $(\Gamma, \tau) {R} (\Theta, \rho)$. 
If $\lnot \uI C \in \Gamma$, $\Box \lnot C \in \Delta$, and $\tau*\seq{C} \subseteq \sigma$, 
then ${\sim}C, \Box \lnot C \in \Theta$ and $\tau*\seq{C} \subseteq \rho$ because $\Delta \prec \Theta$ and $\sigma \subset \rho$. 
Therefore, $(\Delta, \sigma) {S_{(\Gamma, \tau)}} (\Theta, \rho)$. 
By Fact \ref{FC}.4,  $\J{5}$ is valid in $\mathcal{F}$. 

	\item $L = \IL$.  
We can similarly prove that $\J{2}_{+}$ and $\J{5}$ are valid in $\mathcal{F}$ as above. 
Also, $(\Gamma, \tau) {R} (\Delta, \sigma)$ obviously implies $(\Delta, \sigma) {S_{(\Gamma, \tau)}} (\Delta, \sigma)$, and hence $\J{1}$ is valid in $\mathcal{F}$ by Fact \ref{FC}.1. 
\end{itemize}

Let $(W, R, \{S_{x}\}_{x \in W}, \Vdash)$ be an $\IL^-$-model  with 
\[
(\Gamma, \tau) \Vdash p \Longleftrightarrow p \in \Gamma
\]
for any $(\Gamma, \tau) \in W$. 
Then, we prove the following claim. 

\begin{cl}[Truth Lemma]
Let $A' \in \Phi$ and let $(\Gamma, \tau) \in W$. Then, 
\[
A' \in \Gamma \Longleftrightarrow (\Gamma, \tau) \Vdash A'.
\]
\end{cl}

\begin{proof}

We prove by induction on the construction of $A'$. 
We only prove the case $A' \equiv \uI C$. 

($\Longrightarrow$): Suppose $\uI C \in \Gamma$. 
We show $(\Gamma, \tau) \Vdash \uI C$. 
Let $(\Delta, \sigma) \in W$ such that $(\Gamma, \tau) {R} (\Delta, \sigma)$. 
We take $\Theta \in K_{\ell}$ and a sequence $\rho$ of formulas of $\Phi_{\uI}$ for each of the following two cases. 

\begin{itemize}
	\item Suppose $\lnot \uI D \in \Gamma$, $\Box \lnot D \in \Delta$ and $\tau*\seq{D} \subseteq \sigma$. 
Since $\uI C \in \Gamma$,  there exists $\Theta \in K_{\ell}$ such that $C, {\sim} D, \Box \lnot D \in \Theta$ and $\Gamma \prec \Theta$ by Lemma \ref{lem2}.4. 
Then, let $\rho := \tau*\seq{D}$. 

	\item Otherwise. 
Since $\Gamma \prec \Delta$, we have $\Box \bot \notin \Gamma$. 
Hence, $\uI \bot \notin \Gamma$ by $\uJ{6}$. 
Since $\uI C \in \Gamma$, there exists $\Theta \in K_{\ell}$ such that $C \in \Theta$ and $\Gamma \prec \Theta$ by Lemma \ref{lem2}.3. 
Then, let $\rho := \tau*\seq{\bot}$. 
\end{itemize}

In any case, we have $\rank(\Theta) + |\rho| \leq \rank(\Gamma) + |\tau| \leq \rank(\Gamma_{0})$. 
Therefore, we obtain $(\Theta, \rho) \in W$ and $(\Delta, \sigma) {S_{(\Gamma, \tau)}} (\Theta, \rho)$. 
By induction hypothesis, we have $(\Theta, \rho) \Vdash C$, and hence $(\Gamma, \tau) \Vdash \uI C$. 

\vspace{2mm}

($\Longleftarrow$): Suppose $\uI C \notin \Gamma$. 
We show $(\Gamma, \tau) \nVdash \uI C$. 
Let $\sigma := \tau*\seq{C}$. 
We take $\Delta \in K_{\ell}$ for each of the following two cases. 
\begin{itemize}
	\item $\ell = \il^-(\I{2}, \I{3})$: 
Since $\uI C \notin \Gamma$, there exists $\Delta \in K_{\ell}$ such that $\Gamma \prec \Delta$ and $\Box \bot \in \Delta$ by Lemma \ref{lem1}.1. 
Then, $\Box {\sim} C \in \Delta$. 

	\item $\ell = \il$: 
By Proposition \ref{relation}.1, $\il$ proves $\uJ{15}$. 
Thus, there exists $\Delta \in K_{\ell}$ such that $\Gamma \prec \Delta$ and ${\sim} C, \Box \lnot C \in \Delta$ by Lemma \ref{lem1}.3. 
\end{itemize}

In any case, we have $\rank(\Delta) + |\sigma| \leq \rank(\Gamma) + |\tau| \leq \rank(\Gamma_{0})$. 
Thus, $(\Delta, \sigma) \in W$. 
Then, for any $(\Theta, \rho) \in W$ with $(\Delta, \sigma) {S_{(\Gamma, \tau)}} (\Theta, \rho)$, we have ${\sim}C \in \Theta$ by the definition of $S_{(\Gamma, \tau)}$. 
Since $(\Theta, \rho) \nVdash C$ by induction hypothesis, we obtain $(\Gamma, \tau) \nVdash \uI C$. 
\end{proof}

Since $(\Gamma_{0}, \epsilon) \in W$ and ${\sim}A \in \Gamma_{0}$, 
we obtain $(\Gamma_{0}, \epsilon) \nVdash A$ by Claim 2. 
\end{proof}

\subsection{The unary interpretability logics for the sublogics of $\IL$ which are complete with respect to $\ILS$-frames}\label{SubSec:MCwrtGVF}

In this subsection, we consider the eight sublogics of $\IL$ which are complete with respect to $\ILS$-frames. 
By Facts~\ref{p1}.2 and 3, $\IL^-(\J{4_{+}}) \subseteq \IL^-(\J{2}, \J{4}_{+}) \subseteq \IL^-(\J{2}_{+})$. 
Therefore, $\IL^-(\J{2}, \J{4}_{+})$ corresponds to $\il^-(\I{2})$ by Corollary~\ref{cor1}. 
Hence, we just consider the remaining seven logics. 

Firstly, we consider the five logics $\IL^{-}(\J{4})$, $\IL^{-}(\J{4}, \J{5})$, $\IL^{-}(\J{2})$, $\IL^-(\J{1}, \J{4})$, and $\IL^{-}(\J{1}, \J{4}, \J{5})$. 

\begin{thm}\label{ilincomp4}
Let $L \in \{\IL^{-}(\J{4}, \J{5})$, $\IL^{-}(\J{2})$, $\IL^-(\J{1}, \J{4})$, $\IL^{-}(\J{1}, \J{4}, \J{5})\}$ and let $\ell$ be the following logic; 
\begin{eqnarray*}
\ell := \left\{
\begin{array}{ll}
\il^-(\I{4}) & \text{if} \, \, L = \IL^{-}(\J{4}, \J{5}) \, \, \text{or} \, \, L = \IL^{-}(\J{2}),\\
\il^-(\uJ{1}, \I{4}) & \text{if} \, \, L = \IL^-(\J{1}, \J{4}),\\
\il^-(\uJ{15}, \I{4}) & \text{if} \, \, L = \IL^{-}(\J{1}, \J{4}, \J{5}).
\end{array}
\right.
\end{eqnarray*}
Then, for any $\mathcal{L}(\Box, \uI)$ formula $A$, if $L \vdash A$, then $\ell \vdash A$. 
\end{thm}

\begin{proof}
Suppose $\ell \nvdash A$. 
By Fact~\ref{gcomp}, it suffices to show that there exist an $\ILS$-model $(W, R, \{S_{x}\}_{x \in W}, \Vdash)$ in which all axiom schemata in $L$ are valid and $w \in W$ such that $w \nVdash A$. 
Let $\Phi$ be any finite adequate set including $A$ and let $\Gamma_{0} \in K_{\ell}$ with ${\sim}A \in \Gamma_{0}$. 
We define $\mathcal{F} = (W, R, \{S_{x}\}_{x \in W})$ as follows:
\begin{itemize}
	\item $W := \{(\Gamma, B): \Gamma \in K_{\ell}, B \in \Phi_{\uI}\}$; 
	\item $(\Gamma, B) {R} (\Delta, C) :\Longleftrightarrow \Gamma \prec \Delta$; 
	\item $(\Delta, C) {S_{(\Gamma, B)}} V :\Longleftrightarrow \Gamma \prec \Delta$, there exists $(\Theta, D) \in V$ such that $\Gamma \prec \Theta$, and the following conditions that depend on $L$ hold. 
\begin{itemize}
	\item $L = \IL^{-}(\J{4}, \J{5})$: If $\lnot \uI C \in \Gamma$ and $\Box \bot \in \Delta$, then there exists $(\Lambda, E) \in V$ such that ${\sim}C \in \Lambda$. 
	\item $L = \IL^{-}(\J{2})$: If $\lnot \uI C \in \Gamma$, then there exist $(\Lambda_{0}, E), (\Lambda_{1}, C) \in V$ such that ${\sim}C \in \Lambda_{0}$ and $\Gamma \prec \Lambda_{1}$. 
	\item $L = \IL^{-}(\J{1}, \J{4})$: If $\lnot \uI C \in \Gamma$ and ${\sim} C \in \Delta$, then there exists $(\Lambda, E) \in V$ such that ${\sim}C \in \Lambda$. 
	\item $L = \IL^{-}(\J{1}, \J{4}, \J{5})$: If $\lnot \uI C \in \Gamma$ and ${\sim}C, \Box \lnot C \in \Delta$, then there exists $(\Lambda, E) \in V$ such that ${\sim}C \in \Lambda$. 
\end{itemize} 
\end{itemize}

Since $(\Gamma_{0}, \bot) \in W$, we have $W \neq \emptyset$. 
Also, $S_{(\Gamma, B)}$ holds Monotonicity by its definition. 
Thus, $\mathcal{F}$ is obviously a finite $\ILS$-frame. 
We show that all axiom schemata of $L$ are valid in $\mathcal{F}$. 
In any case, if $(\Delta, C) {S_{(\Gamma, B)}} V$, then there exists $(\Theta, D) \in W$ such that $\Gamma \prec \Theta$. 
Hence $\J{4}$ is valid in $\mathcal{F}$ by Fact \ref{gFC}.3. 
We distinguish the following four cases. 

\begin{itemize}
	\item $L = \IL^-(\J{4}, \J{5})$: 
We show that $\J{5}$ is valid in $\mathcal{F}$. 
Assume $(\Gamma, B) {R} (\Delta, C)$ and $(\Delta, C) {R} (\Theta, D)$. 
Then, $\Gamma \prec \Delta$ and $\Gamma \prec \Theta$. 
Also, $\Box \bot \notin \Delta$ because $\Delta \prec \Theta$. 
Therefore, $(\Delta, C) {S_{(\Gamma, B)}} \{(\Theta, D)\}$, and hence $\J{5}$ is valid in $\mathcal{F}$ by Fact~\ref{gFC}.5. 

	\item $L = \IL^-(\J{2})$: 
Suppose $(\Delta, C) {S_{(\Gamma, B)}} V$ and for any $(\Delta', C') \in (V \cap R[(\Gamma, B)])$, $(\Delta', C') {S_{(\Gamma, B)}} U_{(\Delta', C')}$. 
We distinguish the following two cases. 

\begin{itemize}
	\item $\lnot \uI C \in \Gamma$: 
Then, there exists $(\Lambda_{1}, C) \in V$ such that $\Gamma \prec \Lambda_{1}$. 
Since $(\Lambda_{1}, C) \in (V \cap R[(\Gamma, B)])$, we have $(\Lambda_{1}, C) {S_{(\Gamma, B)}} U_{(\Lambda_{1}, C)}$. 
By the definition of $S_{(\Gamma, B)}$, there exist $(\Lambda_{0}', E), (\Lambda_{1}', C) \in U_{(\Lambda_{1}, C)}$ such that ${\sim}C \in \Lambda_{0}'$ and $\Gamma \prec \Lambda_{1}'$. 
Therefore, we have $(\Delta, C) {S_{(\Gamma, B)}} U_{(\Lambda_{1}, C)}$. 
By Monotonicity, we obtain $(\Delta, C) {S_{(\Gamma, B)}}\bigcup_{(\Delta', C') \in V}U_{(\Delta', C')}$. 

	\item Otherwise: 
By the definition of $S_{(\Gamma, B)}$, there exists $(\Theta, D) \in V$ such that $\Gamma \prec \Theta$. 
Since $(\Theta, D) \in (V \cap R[(\Gamma, B)])$, we have $(\Theta, D) {S_{(\Gamma, B)}}U_{(\Theta, D)}$. 
Then, there exists $(\Theta', D') \in U_{(\Theta, D)}$ such that $\Gamma \prec \Theta'$. 
Therefore, we have $(\Delta, C) {S_{(\Gamma, B)}} U_{(\Theta, D)}$. 
We obtain $(\Delta, C) {S_{(\Gamma, B)}} \bigcup_{(\Delta', C') \in V}U_{(\Delta', C')}$ by Monotonicity. 
\end{itemize}

Therefore, $\J{2}$ is valid in $\mathcal{F}$ by Fact~\ref{gFC}.2. 

	\item $L = \IL^-(\J{1}, \J{4})$: 
Obviously, if $(\Gamma, B) {R} (\Delta, C)$, then $(\Delta, C) {S_{(\Gamma, B)}} \{(\Delta, C)\}$. 
Therefore, $\J{1}$ is valid in $\mathcal{F}$ by Fact~\ref{gFC}.1. 

	\item $L = \IL^-(\J{1}, \J{4}, \J{5})$: 
This is proved by a similar way as above.  
\end{itemize}

Let $(W, R, \{S_{x}\}_{x \in W}, \Vdash)$ be an $\ILS$-model with 
\[
(\Gamma, B) \Vdash p \Longleftrightarrow p \in \Gamma
\]
for any $(\Gamma, B) \in W$. 
We prove the following claim. 

\begin{cl}[Truth Lemma]
Let $A' \in \Phi$ and let $(\Gamma, B) \in W$. Then,  
\[
A' \in \Gamma \Longleftrightarrow (\Gamma, B) \Vdash A'.
\]
\end{cl}

\begin{proof}
We prove by induction on the construction of $A'$. 
We only prove the case $A' \equiv \uI C$. 

($\Longrightarrow$): Suppose $\uI C \in \Gamma$. 
We show $(\Gamma, B) \Vdash \uI C$. 
Let $(\Delta, D) \in W$ such that $(\Gamma, B) {R} (\Delta, D)$. 
We distinguish the following two cases. 
\begin{itemize}
	\item Suppose $\lnot \uI D \in \Gamma$. 
Since $\uI C \in \Gamma$ and $\Gamma \prec \Delta$, there exists $\Theta \in K_{\ell}$ such that $\Gamma \prec \Theta$ and $C \in \Theta$ by Lemma \ref{lem3}. 
Also, since $\uI C \in \Gamma$ and $\lnot \uI D \in \Gamma$, there exists $\Lambda \in K_{\ell}$ such that $C, {\sim}D \in \Lambda$ by Lemma \ref{lem2}.1. 
Let $V := \{(\Theta, D), (\Lambda, D)\}$. 
Then, $(\Delta, D) {S_{(\Gamma, B)}} V$.

	\item Otherwise. 
Since $\uI C \in \Gamma$ and $\Gamma \prec \Delta$, there exists $\Theta \in K_{\ell}$ such that $\Gamma \prec \Theta$ and $C \in \Theta$ by Lemma \ref{lem3}. 
Let $V := \{(\Theta, D)\}$. 
Then, $(\Delta, D) {S_{(\Gamma, B)}} V$. 
\end{itemize}
In any case, by induction hypothesis, we have $(\Theta', D') \Vdash C$ for any $(\Theta', D') \in V$. 
Hence, we obtain $(\Gamma, B) \Vdash \uI C$.

($\Longleftarrow$): Suppose $\uI C \notin \Gamma$. 
We show $(\Gamma, B) \nVdash \uI C$. 
We take $\Delta \in K_{\ell}$ for each of the following three cases. 
\begin{itemize}
	\item $\ell = \il^-(\I{4})$: 
There exists $\Delta \in K_{\ell}$ such that $\Gamma \prec \Delta$ and $\Box \bot \in \Delta$ by Lemma \ref{lem1}.1. 

	\item $\ell = \il^-(\uJ{1}, \I{4})$: 
There exists $\Delta \in K_{\ell}$ such that $\Gamma \prec \Delta$ and $ {\sim} C \in \Delta$ by Lemma \ref{lem1}.2. 

	\item $\ell = \il^-(\uJ{15}, \I{4})$: 
There exists $\Delta \in K_{\ell}$ such that $\Gamma \prec \Delta$ and $\Box \lnot C,  {\sim} C \in \Delta$ by Lemma \ref{lem1}.3. 
\end{itemize}

In any cases, we have $(\Gamma, B) {R} (\Delta, C)$. 
Moreover, for any $V \subseteq W$ with $(\Delta, C) {S_{(\Gamma, B)}} V$, there exists $(\Lambda, E)$ such that ${\sim}C \in \Lambda$, and hence $(\Lambda, E) \nVdash C$ by induction hypothesis. 
Therefore, we obtain $(\Gamma, B) \nVdash \uI C$. 
\end{proof}

Since $(\Gamma_{0}, \bot) \in W$ and ${\sim}A \in \Gamma_{0}$, we obtain $(\Gamma_{0}, \bot) \nVdash A$ by Claim 3. 
\end{proof}

By Theorem~\ref{ilincomp4} and the implication $\Longleftarrow$ in Theorem~\ref{MainTheorem}, we obtain the following corollary. 

\begin{cor}\label{cor2}
Let $A$ be any $\mathcal{L}(\Box, \uI)$-formula and let $L$ be any logic such that $\IL^-(\J{4}) \subseteq L \subseteq \IL^-(\J{4}, \J{5})$ or $\IL^-(\J{4}) \subseteq L \subseteq \IL^-(\J{2})$. 
Then, $L \vdash A$ iff $\il^-(\I{4}) \vdash A$. 
\end{cor}

Next, we consider the remaining two logics $\IL^-(\J{2}, \J{5})$ and $\IL^{-}(\J{2}, \J{4}_{+}, \J{5})$.

\begin{thm}\label{ilincomp25245}
Let $L \in \{\IL^-(\J{2}, \J{5}), \IL^{-}(\J{2}, \J{4}_{+}, \J{5})\}$ and let $\ell$ be the following logic; 
\begin{eqnarray*}
\ell := \left\{
\begin{array}{ll}
\il^{-}(\I{4},  \uJ{25}) & \text{if} \, \, L = \IL^-(\J{2}, \J{5}),\\
\il^{-}(\I{2},  \uJ{25}) & \text{if} \, \, L = \IL^{-}(\J{2}, \J{4}_{+}, \J{5}). 
\end{array}
\right.
\end{eqnarray*}
For any $\mathcal{L}(\Box, \uI)$ formula $A$, if $L \vdash A$, then $\ell \vdash A$. 
\end{thm}

\begin{proof}
Suppose $\ell \nvdash A$. 
We find an $\ILS$-model $(W, R, \{S_{x}\}_{x \in W}, \Vdash)$ in which all axiom schemata in $L$ are valid and $w \in W$ with $w \nVdash A$. 
Let $\Phi$ be any finite adequate set including $A$ and let $\Gamma_{0} \in K_{\ell}$ with ${\sim}A \in \Gamma_{0}$. 
For each $\Gamma \in K_{\ell}$, we define $\rank(\Gamma)$ as in the proof of Theorem \ref{ilcomp2a5}. 
Let $r:=\max\{\rank(\Gamma) : \Gamma \in K_{\ell}\}$. 
Then, we define $\mathcal{F} = (W, R, \{S_{x}\}_{x \in W})$ as follows:
\begin{itemize}
	\item $W := \{(\Gamma, \tau): \Gamma \in K_{\ell}$ and $\tau$ is a finite sequence of elements of $\Phi_{\uI}$ with $\rank(\Gamma) + |\tau| \leq r\}$; 
	\item $(\Gamma, \tau) {R} (\Delta, \sigma) :\Longleftrightarrow \Gamma \prec \Delta$ and $\tau \subset \sigma$; 
	\item $(\Delta, \sigma) {S_{(\Gamma, \tau)}} V :\Longleftrightarrow (\Gamma, \tau) {R} (\Delta, \sigma)$, there exists $(\Theta, D) \in W$ such that $\Gamma \prec \Theta$ and the following conditions that depend on $L$ hold. 
\begin{itemize}
	\item $L = \IL^-(\J{2}, \J{5})$: If $\lnot \uI C \in \Gamma$, $\tau*\seq{C} \subseteq \sigma$, and $\Box \lnot C \in \Delta$, then there exist $(\Lambda_{1}, \rho_{1}), (\Lambda_{2}, \rho_{2}) \in V$ such that ${\sim}C \in \Lambda_{1}, \Gamma \prec \Lambda_{2}, \Box \lnot C \in \Lambda_{2}$, and $\tau*\seq{C} \subseteq \rho_{2}$.
 
	\item $L = \IL^-(\J{2}, \J{4}_{+}, \J{5})$: If $\lnot \uI C \in \Gamma$, $\tau*\seq{C} \subseteq \sigma$ and $\Box \lnot C \in \Delta$, then there exist $(\Lambda_{1}, \rho_{1}), (\Lambda_{2}, \rho_{2}) \in V$ such that ${\sim}C \in \Lambda_{1}, \Gamma \prec \Lambda_{1}, \tau*\seq{C} \subseteq \rho_{1}, \Gamma \prec \Lambda_{2}, \Box \lnot C \in \Lambda_{2}$, and $\tau*\seq{C} \subseteq \rho_{2}$. 
\end{itemize} 
\end{itemize}

Since $(\Gamma_{0}, \epsilon) \in W$, we have $W \neq \emptyset$. 
Thus, $\mathcal{F}$ is obviously a finite $\ILS$-frame. 
We show that all axiom schemata of $L$ are valid in $\mathcal{F}$. 
In any case, $\J{4}$ is valid in $\mathcal{F}$ by Fact \ref{gFC}.3. 
We distinguish the following two cases. 

\begin{itemize}
	\item $L = \IL^-(\J{2}, \J{5})$:  
Firstly, we show that $\J{2}$ is valid in $\mathcal{F}$. 
Suppose $(\Delta, \sigma) {S_{(\Gamma, \tau)}} V$ and for any $(\Delta', \sigma') \in (V \cap R[(\Gamma, \tau)])$, $(\Delta', \sigma') {S_{(\Gamma, \tau)}} U_{(\Delta', \sigma')}$. 
We distinguish the following two cases. 
\begin{itemize}
	\item $\lnot \uI C \in \Gamma$, $\tau*\seq{C} \subseteq \sigma$, and $\Box \lnot C \in \Delta$: 
Then, there exists $(\Lambda_{2}, \rho_{2}) \in V$ such that $\Gamma \prec \Lambda_{2}, \Box \lnot C \in \Lambda_{2}$, and $\tau*\seq{C} \subseteq \rho_{2}$. 
Since $(\Lambda_{2}, \rho_{2}) \in (V \cap R[(\Gamma, \tau)])$, $(\Lambda_{2}, \rho_{2}) {S_{(\Gamma, \tau)}} U_{(\Lambda_{2}, \rho_{2})}$. 
Thus, there exist $(\Lambda_{1}', \rho_{1}'), (\Lambda_{2}', \rho_{2}') \in U_{(\Lambda_{2}, \rho_{2})}$ such that ${\sim}C \in \Lambda_{1}'$, $\Gamma \prec \Lambda_{2}', \Box{\sim}C \in \Lambda_{2}'$, and $\tau*\seq{C} \subseteq \rho_{2}'$. 
Therefore, $(\Delta, \sigma) {S_{(\Gamma, \tau)}} U_{(\Lambda_{2}, \rho_{2})}$, and hence, by Monotonicity, we obtain $(\Delta, \sigma) {S_{(\Gamma, \tau)}} \bigcup_{(\Delta', \sigma') \in V} U_{(\Delta', \sigma')}$. 

	\item Otherwise: 
By the definition of $S_{(\Gamma,\tau)}$, there exists$(\Theta, \rho) \in V$ such that $(\Gamma,\tau) {R} (\Theta, \rho)$. 
Since $(\Theta, \rho) \in (V \cap R[(\Gamma, \tau)])$, $(\Theta, \rho) {S_{(\Gamma, \tau)}}U_{(\Theta, \rho)}$. 
Then, there exists $(\Theta', \rho') \in U_{(\Theta, \rho)}$ such that $(\Gamma,\tau) {R} (\Theta', \rho')$. 
Therefore, $(\Delta, \sigma) {S_{(\Gamma, \tau)}} U_{(\Theta, \rho)}$, and hence $(\Delta, \sigma) {S_{(\Gamma, \tau)}} \bigcup_{(\Delta', \sigma') \in V} U_{(\Delta', \sigma')}$ by Monotonicity. 
\end{itemize}
Therefore, $\J{2}$ is valid in $\mathcal{F}$ by Fact \ref{gFC}.2. 

Secondly, we show that $\J{5}$ is valid in $\mathcal{F}$. 
Assume $(\Gamma, \tau) {R} (\Delta, \sigma)$ and $(\Delta, \sigma) {R} (\Theta, \rho)$. 
Then, $(\Gamma, \tau) {R} (\Theta, \rho)$.
Also, since $\Delta \prec \Theta$ and $\sigma \subset \rho$, if $\lnot \uI C \in \Gamma$, $\tau*\seq{C} \subseteq \sigma$, and $\Box \lnot C \in \Delta$, 
then ${\sim}C, \Box \lnot C \in \Theta$ and $\tau*\seq{C} \subseteq \rho$. 
Therefore, $(\Delta, \sigma) {S_{(\Gamma, \tau)}} \{(\Theta, \rho)\}$ and by Fact \ref{gFC}.5,  $\J{5}$ is valid in $\mathcal{F}$. 

	\item $L = \IL^-(\J{2}, \J{4}_{+}, \J{5})$.  
We can similarly prove that $\J{2}$ and $\J{5}$ are valid in $\mathcal{F}$ as above. 
Also, $(\Delta, \sigma) {S_{(\Gamma, \tau)}} V$ obviously implies $(\Delta, \sigma) {S_{(\Gamma, \tau)}} (V \cap R[(\Gamma, \tau)])$ by the definition of $S_{(\Gamma, \tau)}$. 
Therefore, $\J{4}_{+}$ is valid in $\mathcal{F}$ by Fact~\ref{gFC}.4. 
\end{itemize}

Let $(W, R, \{S_{x}\}_{x \in W}, \Vdash)$ be an $\IL^-$-model  with 
\[
(\Gamma, \tau) \Vdash p \Longleftrightarrow p \in \Gamma
\]
for any $(\Gamma, \tau) \in W$. 
We prove the following claim. 

\begin{cl}[Truth Lemma]
Let $A' \in \Phi$ and let $(\Gamma, \tau) \in W$. Then, 
\[
A' \in \Gamma \Longleftrightarrow (\Gamma, \tau) \Vdash A'. 
\]
\end{cl}

\begin{proof}

We prove by induction on the construction of $A'$. 
We only prove the case $A' \equiv \uI C$. 

($\Longrightarrow$): Suppose $\uI C \in \Gamma$. 
We show that $(\Gamma, \tau) \Vdash \uI C$. 
Let $(\Delta, \sigma) \in W$ such that $(\Gamma, \tau) {R} (\Delta, \sigma)$. 
We take $V \subseteq W$ for each of the following two cases. 

\begin{itemize}
	\item $\lnot \uI D \in \Gamma$, $\tau*\seq{D} \subseteq \sigma$, and $\Box \lnot D \in \Delta$:  
Since $\uI C \in \Gamma$ and $\lnot \uI D \in \Gamma$, there exists $\Lambda_{1} \in K_{\ell}$ such that $C, {\sim}D \in \Lambda_{1}$ by Lemma \ref{lem2}.1. 
If $\ell = \il^-(\I{2}, \uJ{25})$, then $\Gamma \prec \Lambda_{1}$ also holds by Lemma \ref{lem2}.3. 
On the other hand, there exists $\Lambda_{2} \in K_{\ell}$ such that $\Gamma \prec \Lambda_{2}$ and $C, \Box \lnot D \in \Lambda_{2}$ by Lemma \ref{lem2}.2. 
Then, let $\rho_{1}$ be $\epsilon$ if $\ell = \il^-(\I{4}, \uJ{25})$, and $\tau*\seq{D}$ if $\ell = \il^-(\I{2}, \uJ{25})$, and let $\rho_{2} = \tau*\seq{D}$. 
It is easy to see that $(\Lambda_{1}, \rho_{1}), (\Lambda_{2}, \rho_{2}) \in W$ and we define $V := \{(\Lambda_{1}, \rho_{1}), (\Lambda_{2}, \rho_{2})\}$. 

	\item Otherwise:  
Since $\uI C \in \Gamma$, $\Gamma \prec \Delta$, and $\ell \vdash \I{4}$ by Proposition~\ref{relation}.5, there exists $\Theta \in K_{\ell}$ such that $\Gamma \prec \Theta$ and $C \in \Theta$ by Lemma \ref{lem3}. 
Let $\rho:=\tau*\seq{\bot}$. 
Then, we obtain $(\Theta, \rho) \in W$. 
Let $V := \{(\Theta, \rho)\}$. 
\end{itemize}

In any case, $(\Delta, \sigma) {S_{(\Gamma, \tau)}} V$ by the definition of $S_{(\Gamma, \tau)}$.
Furthermore, for any $(\Theta', \rho') \in V$, $(\Theta', \rho') \Vdash C$ by induction hypothesis. 	
Therefore, we obtain $(\Gamma, \tau) \Vdash \uI C$.

($\Longleftarrow$): Suppose $\uI C \notin \Gamma$. 
We show $(\Gamma, \tau) \nVdash \uI C$. 
Since $\uI C \notin \Gamma$,  there exists $\Delta \in K_{\ell}$ such that $\Gamma \prec \Delta$ and $\Box \bot \in \Delta$ by Lemma \ref{lem1}.1, and hence $\Box \lnot C \in \Delta$. 
Let $\sigma:=\tau*\seq{C}$. 
Then, $(\Delta, \sigma) \in W$ and $(\Gamma, \tau) {R} (\Delta, \sigma)$. 
For any $V \subseteq W$ with $(\Delta, \sigma) {S_{(\Gamma, \tau)}} V$, there exists $(\Lambda_{1}, \rho_{1}) \in V$ such that ${\sim}C \in \Lambda_{1}$ by the definition of $S_{(\Gamma, \tau)}$. 
Since $(\Lambda_{1}, \rho_{1}) \nVdash C$ by induction hypothesis, we obtain $(\Gamma, \tau) \nVdash \uI C$. 
\end{proof}

Since $(\Gamma_{0}, \epsilon) \in W$ and ${\sim}A \in \Gamma_{0}$, 
we obtain $(\Gamma_{0}, \epsilon) \nVdash A$ by Claim 4. 
\end{proof}

This ends the proof of Theorem~\ref{MainTheorem}.

\section{Conclusions and future work}\label{Sec:ConFu}

In this paper, for each of the twenty sublogics of $\IL$, we find the unary interpretability logic corresponding to that logic. 
Then, we understand that several situations are also simplified and changed for these unary interpretability logics compared with the sublogics of $\IL$. 
We find out that the axiom schemata $\J{2}$, $\J{2}_{+}$, and $\J{5}$ may collapse for unary interpretability logics as we stated in Remark~\ref{Rem:Collpse}. 
Moreover, a situation for the completeness with respect to relational semantics may be changed. 
The logic $\IL^-(\J{2}, \J{4}_{+})$ is incomplete with respect to $\IL^- \!$-frame. 
On the other hand, $\il^-(\I{2})$ corresponds to $\IL^-(\J{2}, \J{4}_{+})$, and is complete with respect to $\IL^- \!$-frame because $\il^-(\I{2})$ also corresponds to $\IL^-(\J{4}_{+})$. 

In the rest of this section, we propose several problems from two view points. 

\subsection{Logics of $\Gamma$-conservativity}

In~\cite{Ign91}, Ignatiev investigated the logic of $\Gamma$-conservativity. 
Then, Ignatiev introduced the logic $\CL$ and extensions $\mathbf{SbCLM}$ and $\mathbf{SCL}$ of $\CL$. 
In particular, Ignatiev proved that the logics $\mathbf{SbCLM}$ and $\mathbf{SCL}$ are logics of $\Pi_{2}$-conservativity and $\Gamma$-conservativity ($\Gamma \in \{\Sigma_{n}, \Pi_{n} : n \geq 3\}$), respectively. 
On the other hand, he proposed a problem what are the logics of $\Sigma_{1}$, $\Sigma_{2}$-conservativity. 
Ignatiev's problem remains open. 

We may expect finding the unary $\Sigma_{1}$, $\Sigma_{2}$-conservativity logics. 
Then, we propose the following problem:

\begin{prob}\label{uS12}
What are the unary $\Sigma_{1}$, $\Sigma_{2}$-conservativity  logics? 
\end{prob}

Concerning Problem~\ref{uS12}, it may make sense to investigate unary counterparts of the logics $\mathbf{SbCLM}$ and $\mathbf{SCL}$. 
Then, we also propose the following problem:

\begin{prob}
What are the unary counterparts of the logics $\mathbf{SbCLM}$ and $\mathbf{SCL}$? 
\end{prob}

By Theorem~\ref{MainTheorem}, the set of $\CL$-provable $\mathcal{L}(\Box, \uI)$-formulas collapses to the set of $\IL^-(\J{1}, \J{4}_{+})$-provable one, and $\il^-(\uJ{1}, \I{2})$ is the unary counterpart of $\CL$. 
Hence, we may expect that $\mathbf{SbCLM}$ and $\mathbf{SCL}$ collapse as $\CL$ and something simple unary counterpart of these logics are found.

\subsection{The fixed point and the Craig interpolation properties}

We concern the \textit{fixed point property} (FPP) and the \textit{Craig interpolation property} (CIP) for unary interpretability logics. 
In~\cite{IKO20}, Iwata, Kurahashi, and Okawa investigated FPP and CIP for the twelve sublogics of $\IL$ which are complete with respect to $\IL^-\!$-frames (The results are given at the left side in Table 2).

\begin{defn}
Let $p$ be any propositional variable and let $A$ be any $\mathcal{L}(\Box, \rhd)$-formula. 
We say that $p$ is \textit{modalized} in $A$ if all occurrences of $p$ in $A$ are in the scope of some modal operators $\Box$ or $\rhd$. 
\end{defn}

\begin{defn}
For each $\mathcal{L}(\Box, \rhd)$-formula $A$, let $\mathrm{var}(A)$ be the set of all propositional variables contained in $A$. 
\end{defn}

\begin{defn}\label{Def:FPP}
A logic $L$ is said to have FPP if for any $\mathcal{L}(\Box, \rhd)$-formula $A(p)$ in which $p$ is modalized, there exists an $\mathcal{L}(\Box, \rhd)$-formula $F$ such that $\mathrm{var}(F) \subseteq \mathrm{var}\bigl(A(p)\bigr) \setminus \{p\}$ and $L \vdash F \leftrightarrow A(F)$. 
We say that such a formula $F$ \textit{is a fixed point} of $A(p)$. 
Also, a logic $\ell$ is said to have FPP($\uI$) if FPP holds for $\ell$ with respect to $\mathcal{L}(\Box, \uI)$-formulas. 
\end{defn}

Iwata, Kurahashi, and Okawa~\cite{IKO20} proved that $A(\top) \rhd B\bigl(\Box \lnot A(\top)\bigr)$ is a fixed point of $A(p) \rhd B(p)$ for $\IL^-(\J{2}_{+}, \J{5})$. 
Thus, a fixed point of $\top \rhd B(p)$ is $\top \rhd B(\Box \bot)$. 
Then, as usual (See~\cite{DeJVis91, Lin96}), for any $\mathcal{L}(\Box, \uI)$-formula $A(p)$ in which $p$ is modalized, we can construct a fixed point $F$ of $A(p)$ as an $\mathcal{L}(\Box, \uI)$-formula. 
Therefore, $\IL^-(\J{2}_{+}, \J{5})$ has FPP($\uI$), and hence we obtain the following corollary by Theorem~\ref{MainTheorem}:
\begin{cor}
The logic $\il^-(\I{2}, \I{3})$ has FPP$(\uI)$. 
\end{cor}
Also, by a result in~\cite{IKO20}, the $\mathcal{L}(\Box, \uI)$-formula $\I \lnot p$ does not have a fixed point in $\IL^-(\J{1}, \J{4}_{+}, \J{5})$. 
Therefore, we obtain the following corollary by Theorem~\ref{MainTheorem}:
\begin{cor}\label{nFPPu}
The logic $\il^-(\uJ{15}, \I{2})$ does not have FPP$(\uI)$. 
\end{cor}

Thus, for the fixed point properties, we understand that there are no differences between the twelve sublogics of $\IL$ and the corresponding unary interpretability logics (See Table~2). 
In this context, we propose the following problem:

\begin{prob}
Are there an interpretability logic $L$ and a corresponding unary interpretability logic $\ell$ such that $L$ does not have FPP and $\ell$ has FPP$(\uI)$? 
\end{prob}

Finally, we concern the Craig interpolation properties. 

\begin{defn}
A logic $L$ is said to have CIP if for any $\mathcal{L}(\Box, \rhd)$-formulas $A$ and $B$, there exists an $\mathcal{L}(\Box, \rhd)$-formula $C$ such that $\mathrm{var}(C) \subseteq \mathrm{var}(A) \cap \mathrm{var}(B)$, $L \vdash A \to C$, and $L \vdash C \to B$. 
Also, a logic $\ell$ is said to have CIP($\uI$) if CIP holds for $\ell$ with respect to $\mathcal{L}(\Box, \uI)$-formulas. 
\end{defn}

By a result in~\cite{IKO20}, the uniqueness of fixed points $\bigl($UFP($\uI$)$\bigr)$ holds for the logic $\IL^-(\J{4}_{+})$. 
That is, for any $\mathcal{L}(\Box, \uI)$-formula $A(p)$ in which $p$ is modalized, 
\[
\IL^-(\J{4}_{+}) \vdash \boxdot\bigl(p \leftrightarrow A(p)\bigr) \land \boxdot\bigl(q \leftrightarrow A(q)\bigr) \to (p \leftrightarrow q). 
\]
Thus, UFP($\uI$) holds for extensions of $\il^-(\I{2})$ by Theorem~\ref{MainTheorem}. 
It is easily seen that CIP($\uI$) implies FPP($\uI$) for extensions of $\il^-(\I{2})$ (See~\cite{IKO20}). 
Therefore, we obtain the following corollary by Corollary~\ref{nFPPu}. 
\begin{cor}\label{CIPu}
CIP$(\uI)$ does not hold for $\il^-(\I{2})$, $\il^-(\uJ{1}, \I{2})$, and $\il^-(\uJ{15}, \I{2})$. 
\end{cor}
We do not know whether CIP($\uI$) holds or not for several logics which are not extensions of $\il^-(\I{2})$ or  have FPP($\uI$). 

\begin{prob}
Does CIP$(\uI)$ hold for $\il^-$, $\il^-(\uJ{1})$, $\il^-(\uJ{15})$, and $\il^-(\I{2}, \I{3})$?
\end{prob}

\begin{table}[h]\label{Tab2}
\centering
\begin{tabular}[t]{|l|c|c|}
\hline
 & FPP & CIP\\
\hline
\hline
$\IL^-$ & $\times$~\cite{IKO20} & $\times$~\cite{IKO20}\\
\cline{1-3}
$\IL^-(\J{5})$ & $\times$~\cite{IKO20} & $\times$~\cite{IKO20}\\
\hline
$\IL^-(\J{1})$ & $\times$~\cite{IKO20} & $\times$~\cite{IKO20}\\
\hline
$\IL^-(\J{4}_+)$ & $\times$~\cite{IKO20} & $\times$~\cite{IKO20}\\
\cline{1-3}
$\IL^-(\J{4}_+, \J{5})$ & $\times$~\cite{IKO20} & $\times$~\cite{IKO20}\\
\cline{1-3}
$\IL^-(\J{2}_+)$ & $\times$~\cite{IKO20} & $\times$~\cite{IKO20}\\
\hline
$\IL^-(\J{1}, \J{5})$ & $\times$~\cite{IKO20} & $\times$~\cite{IKO20}\\
\hline
$\IL^-(\J{1}, \J{4}_+)$ & $\times$~\cite{IKO20} & $\times$~\cite{IKO20}\\
\cline{1-3}
$\CL$ & $\times$~\cite{IKO20} & $\times$~\cite{IKO20}\\
\hline
$\IL^-(\J{1}, \J{4}_+, \J{5})$ & $\times$~\cite{IKO20} & $\times$~\cite{IKO20}\\
\hline
$\IL^-(\J{2}_+, \J{5})$ & \checkmark~\cite{IKO20} & \checkmark~\cite{IKO20}\\
\hline
$\IL$ & \checkmark~\cite{DeJVis91} & \checkmark~\cite{AHD01}\\ 
\hline
\end{tabular}
\begin{tabular}[t]{|l|c|c|}
\hline
 & FPP($\uI$) & CIP($\uI$) \\
\hline
\hline
$\il^-$ &  $\times$ & ?\\
 & &\\
\hline
 $\il^{-}(\uJ{1})$ & $\times$ & ?\\
\hline
 $\il^{-}(\I{2})$ & $\times$ & $\times$\\
 &  & \\
 &  & \\
\hline
 $\il^{-}(\uJ{15})$ & $\times$ & ?\\
\hline
 $\il^{-}(\uJ{1}, \I{2})$ & $\times$ & $\times$\\
 &  & \\
\hline
$\il^{-}(\uJ{15}, \I{2})$ & $\times$ & $\times$\\
\hline
$\il^{-}(\I{2}, \I{3})$ & \checkmark  & ?\\
\hline
 $\il^-(\uJ{1}, \I{2}, \I{3})$  & \checkmark~\cite{deR92} & \checkmark~\cite{deR92} \\ 
\hline
\end{tabular}
\caption{FPP, CIP, FPP($\uI$), and CIP($\uI$)}
\end{table}

\section*{Acknowledgement}

The author would like to thankful to Taishi Kurahashi for his helpful and valuable comments.  
The research was supported by Foundation of Research Fellows, The Mathematical Society of Japan. 

\bibliographystyle{plain}
\bibliography{ref}

\end{document}